\documentclass[10pt]{article}%
\usepackage{amsmath}
\usepackage{amsfonts}
\usepackage{amssymb}
\usepackage{graphicx}%
\usepackage{mathrsfs}
\usepackage[square, comma, sort&compress, numbers]{natbib}
\usepackage{amssymb, color}
\usepackage[body={15.5cm,21cm}, top=3cm]{geometry}%
\providecommand{\U}[1]{\protect \rule{.1in}{.1in}}

\newtheorem{theorem}{Theorem}[section]

\newtheorem{corollary}[theorem]{Corollary}

\newtheorem{definition}[theorem]{{Definition}}
\newtheorem{example}[theorem]{{ Example}}

\newtheorem{lemma}[theorem]{Lemma}

\newtheorem{remark}[theorem]{{Remark}}

\newenvironment{proof}[1][Proof]{\noindent \textbf{#1.} }{\  \rule{0.5em}{0.5em}}
\begin{document}

\title{Probabilistic approach to singular perturbations of viscosity solutions to  nonlinear parabolic PDEs}
\author{Mingshang Hu \thanks{Zhongtai Securities Institute for Financial Studies, Shandong University. humingshang@sdu.edu.cn. Research supported by the National Natural Science Foundation of China (No. 11671231) and the Qilu Young Scholars Program of Shandong University.}
\and Falei Wang\thanks{Zhongtai Securities Institute for Financial  Studies and School of Mathematics, Shandong University.
flwang2011@gmail.com (Corresponding author). Research supported by   the National Natural Science Foundation of China (No. 12031009, No. 11601282 and No. 11871310) and the Young Scholars Program of Shandong University.
Hu and Wang's research was
partially supported by the National Key R\&D Program of China (No. 2018YFA0703900).}}
\date{}
\maketitle
\begin{abstract}
In this paper, we  prove  a convergence theorem for singular perturbations problems for a class of fully nonlinear parabolic  partial differential equations (PDEs) with ergodic  structures. The limit function is represented as the viscosity solution to a fully nonlinear degenerate
 PDEs. Our approach is mainly based on  $G$-stochastic analysis argument. As a byproduct, we also establish  the averaging principle for stochastic differential equations driven by $G$-Brownian motion  ($G$-SDEs) with two time-scales.  The results extend  Khasminskii's averaging principle to  nonlinear  case.
\end{abstract}

\textbf{Key words}: singular perturbation, averaging  principle,  nonlinear  PDE, $G$-Brownian motion

\textbf{MSC-classification}: 60H10, 60H30

\section{Introduction}
The present paper is devoted to the research of singular perturbations  for a class of fully nonlinear degenerate parabolic PDEs with ergodicity coefficients. Our main tool is the nonlinear stochastic analysis theory formulated by Peng \cite{P07a}. Indeed, we shall  investigate the singular perturbation problems
through  asymptotic analysis of SDEs with slow and fast time-scales in the $G$-expectation framework.

 In this framework, Peng systemically established the nonlinear stochastic calculus theory, such as $G$-Brownian motion,
$G$-stochastic integral  and so on. Due to this nonlinear structure, the $G$-expectation theory provides a useful tool for the research of fully nonlinear PDEs and  volatility ambiguity in finance. Indeed, Song \cite{Song14}
obtained gradient estimates for a class of fully nonlinear PDEs by coupling methods for $G$-diffusion processes,
Biagini et al. \cite{BF} studied robust mean-variance hedging,
and Fouque, Pun and  Wong \cite{FP2} considered the asset allocation problem among a risk-free asset and two risky assets with an ambiguous correlation between the two risky assets.
A notion quite related  to $G$-expectation is the second order BSDE (2BSDE) framework proposed   by Soner, Touzi and Zhang \cite{STZ1}. Indeed, the setting of 2BSDE is more general than that of $G$-expectation, whereas $G$-expectation has more regularity, see \cite{HJPS, Lin,  PTZ, STZ, S2} and the references
therein for more research on this field.

In the present article, we shall consider averaging principle for the following $G$-SDEs with rapidly varying coefficients: for each $x=(\tilde{x},\bar{x})\in\mathbb{R}^n\times\mathbb{R}^{n}$ and $0< \varepsilon<1$,
\begin{align}\label{myw201}
\begin{cases}
&{\displaystyle \widetilde{X}^{\varepsilon,x}_t=\tilde{x}+\int^t_0\widetilde{b}(\widetilde{X}^{\varepsilon,x}_s,\overline{X}^{\varepsilon,x}_s)ds+\sum\limits_{i,j=1}^d\int^t_0\widetilde{h}_{ij}(\widetilde{X}^{\varepsilon,x}_s,\overline{X}^{\varepsilon,x}_s)d\langle B^i,B^j\rangle_s+\int^t_0\widetilde{\sigma}(\widetilde{X}^{\varepsilon,x}_s,\overline{X}^{\varepsilon,x}_s)dB_s,}\\
&{ \displaystyle \overline{X}^{\varepsilon,x}_t=\bar{x}+\int^t_0\frac{\overline{b}(\widetilde{X}^{\varepsilon,x}_s,\overline{X}^{\varepsilon,x}_s)}{\varepsilon}ds+\sum\limits_{i,j=1}^d\int^t_0\frac{\overline{h}_{ij}(\widetilde{X}^{\varepsilon,x}_s,\overline{X}^{\varepsilon,x}_s)}{\varepsilon}d\langle B^i,B^j\rangle_s+\int^t_0\frac{\overline{\sigma}(\widetilde{X}^{\varepsilon,x}_s,\overline{X}^{\varepsilon,x}_s)}{\sqrt{\varepsilon}}dB_s,}
\end{cases}
\end{align}
where $\widetilde{b}$, $\overline{b}$, $\widetilde{h}_{ij}=\widetilde{h}_{ji}$, $\overline{h}_{ij}=\overline{h}_{ji}:\mathbb{R}^{2n}\rightarrow \mathbb{R}^{n}$, $\widetilde{\sigma},\overline{\sigma}:\mathbb{R}^{2n}\rightarrow \mathbb{R}^{n\times d}$ are deterministic non-periodic
functions. Here the parameter $\varepsilon$ is used to describe the ratio of time scale between the diffusion processes $\widetilde{X}^{\varepsilon,x}$ and $\overline{X}^{\varepsilon,x}$. Then, with this time scale $\widetilde{X}^{\varepsilon,x}$ is referred as slow component and $\overline{X}^{\varepsilon,x}$ as the fast component, respectively.
 Note that the distribution of the slow component can be represented as  the viscosity solution to the following fully nonlinear PDE (see section 2.2):
\begin{equation}
\left \{
\begin{array}
[c]{l}%
\partial_{t}u^{\varepsilon}-G\left((\sigma^{\varepsilon})^{\top}D^{2}u^{\varepsilon}\sigma+2[\langle h^{\varepsilon}_{ij},Du^{\varepsilon}\rangle]_{i,j=1}^d\right)-\langle b^{\varepsilon},Du^{\varepsilon}\rangle=0,\ (t,\tilde{x},\bar{x})\in(0,T)\times\mathbb{R}^n\times\mathbb{R}^{n},\\
u^{\varepsilon}(0,\tilde{x},\bar{x})=\varphi(\tilde{x}),\ (\tilde{x},\bar{x})\in\mathbb{R}^n\times\mathbb{R}^n.
\end{array}
\right.  \label{myw990}%
\end{equation}
Then, our aim is to describe the limit behaviour of fully nonlinear PDE \eqref{myw990} as $\varepsilon\rightarrow 0$ through
  averaging of the  $G$-SDEs \eqref{myw201} under some appropriate assumptions.

The averaging principle for SDEs was first studied by Khasminskii \cite{Kh} in 1968. Under some ergodicity assumptions,  Khasminskii proved that the slow diffusion process converges weakly to the so-called  averaged SDE,
whose coefficients are characterized by  integrals
with respect to the related invariant probability measure.
 Since then, the averaging principle for diffusion processes have been studied with great interest and moreover, it provides a powerful tool for
the research of  singular perturbation problems for  linear parabolic PDEs.
In particular, Khasminskii and Krylov \cite{Kh2} established the averaging principle  for diffusion processes in non-ergodic case, and obtained
the limit behavior of linear parabolic PDEs without ergodic coefficients. We refer the reader to  \cite{CF,FW,PS, PV1, RX} and the references therein  for more research on this topic.

There is also a vast literature on the singular perturbations of nonlinear PDEs based on probabilistic argument. With the help of backward stochastic differential equations (BSDEs), Buckdahn and Hu  \cite{BHP} studied homogenization of viscosity solutions to semilinear parabolic PDEs with periodic structures, and Bahlali, Elouaflin and Pardoux \cite{BE1, BE} extended  the results of \cite{Kh2} to  semilinear parabolic PDEs.
 In \cite{BI1}, Buckdahn and Ichihara considered  homogenization of fully nonlinear parabolic PDEs in periodic case by stochastic control approach.
For more research on this field, we refer the reader to  \cite{BHP1,DF,GT,PV} and the references therein.

Motivated by the seminal work \cite{Kh}, we  shall study the  limit behaviour of the slow $G$-diffusion process in the ergodic case. However, due to the nonlinearity of  $G$-expectation, the averaged $G$-SDE has more complicated structure,
which cannot be described by averaging the coefficients of the slow component.  In this case, the averaged coefficients will interact with each other, and
cannot be identified separately  (see Lemma \ref{myw123}). On the other hand, the invariant expectations may not coincide with the ergodic expectations  in the $G$-expectation framework (cf. \cite{HW}).
To overcome these difficulties, we shall combine nonlinear stochastic calculus and viscosity solution techniques  to analyze the limit distribution of the slow component.
In particular, we shall construct the averaged PDE through the ergodic BSDEs approach in the $G$-expectation framework.
Indeed,  the ergodic BSDEs introduced by Fuhrman,  Hu and Tessitore \cite{FH} provides a useful method for the study of large time behaviour of parabolic PDEs (cf. \cite{HM,HM2}).

For our purpose, we first give a priori estimate of the $G$-SDE \eqref{myw201} under the disspativity condition.
This  is crucial for the equicontinuity of the viscosity solutions to  the  PDEs \eqref{myw990} with fast varying coefficients.
Then, according to the Arzel\`{a}-Ascoli theorem, we could  build a convergent subsequence of  the viscosity solutions.
Finally,  with the help of dynamic programming principle and  Khasminskii's discretization approach, we show that the limit function is the unique viscosity solution to the averaged PDE, which is a fully nonlinear PDE independent of the argument $\bar{x}$.
Moreover, in the spirit of the Markov property, we  could obtain that the limit of finite dimensional distribution is also determined by the averaged PDE, which together with  Kolmogorov's criterion for weak compactness  implies  Khasminskii's averaging principle. In addition, we also extend the $G$-SDEs \eqref{myw201} with two time-scales to a
more general case.

In conclusion, we develop an alternative method for the research of averaging of SDEs and singular perturbations of fully nonlinear parabolic PDEs without periodic structures. In particular, we extend the  one of \cite{Kh} to  a class of fully nonlinear PDEs through $G$-stochastic analysis methods.
For a closest related research, we refer the reader to \cite{A1} and  the references therein. In \cite{A1},
 Alvarez and Bardi  used the so-called perturbed test function method to study more general fully nonlinear PDEs with periodic coefficients.

The paper is organized as follows. In section 2, we introduce the $G$-SDEs with slow and fast time-scales. Then, we state the main  results in section 3.
Section 4 is devoted to the proofs of two main theorems. In section 5, we give an extended case.

\section{Formulation of the problem}

In this paper, for each Euclidian space, we denote by $\left\langle\cdot,\cdot\right\rangle$ and $|\cdot|$ its scalar product and the
associated norm, respectively. For a given set of parameters $\alpha$, $C(\alpha)$ will denote a positive constant only depending on these parameters and  may change from line to line.

\subsection{The Probabilistic Setup}
Let $\Omega=C_{0}^{d}(\mathbb{R}^{+})$ be the space of all $\mathbb{R}^{d}%
$-valued continuous paths $(\omega_{t})_{t\geq0}$ starting from origin,
equipped with the locally uniform norm. For each $t\in \lbrack0,\infty)$, we define $\Omega_t:=\{\omega_{\cdot\wedge t}:\omega\in\Omega\}$ and
\begin{align*}
 L_{ip}(\Omega_t):=\left\{ \varphi(B_{t_{1}},\ldots,B_{t_{k}}):k\in \mathbb{N}%
,t_{1},\ldots,t_{k}\in \lbrack0,t],\varphi \in C_{b.lip}(\mathbb{R}%
^{k\times d })\right\}, \ \ L_{ip}(\Omega):=\cup_{T\geq 0} L_{ip}(\Omega_T),
\end{align*}
where $C_{b.lip}(\mathbb{R}^{k\times d})$ denotes the space of all
bounded and
Lipschitz functions on $\mathbb{R}^{k\times d}$. Then, for each given monotonic and sublinear function $G:\mathbb{S}%
(d)\rightarrow \mathbb{R}$, Peng \cite{P07a} constructed a  sublinear expectation space $(\Omega,L_{ip}(\Omega),\mathbb{\hat{E}%
},(\mathbb{\hat{E}%
}_t)_{t\geq 0})$, called $G$-expectation space, where $\mathbb{S}(d)$ denotes the space of all $d\times d$ symmetric matrices. The
canonical process $B_t(\omega) = \omega_t$
 is called $G$-Brownian motion.

Indeed, for each $\xi\in L_{ip}(\Omega)$ with the  form of
\begin{equation*}
\xi=\varphi(B_{t_{1}},B_{t_{2}},\cdots,B_{t_{k}}),\  \
0=t_{0}<t_{1}<\cdots<t_{k}\leq T,
\end{equation*}
and for each $t\in \lbrack t_{i-1},t_{i})$, $i=1,\ldots,k$, we  define the conditional  $G$-expectation by
\begin{equation*}
\mathbb{\hat{E}}_{t}[\xi]:=u_{i}(t,B_t;B_{t_{1}},\cdots,B_{
t_{i-1}}).
\end{equation*}
Here,  the function $u_{i}(t,x;x_{1},\cdots,x_{i-1})$ with parameters $(x_{1},\cdots,x_{i-1})\in \mathbb{R}^{(i-1)\times d}$
is the viscosity solution of the following $G$-heat equation:
\begin{equation*}
\partial_{t}u_{i}(t,x;x_{1},\cdots,x_{i-1})+G(\partial^2_{xx}u_{i}(t,x;x_{1},\cdots,x_{i-1}))=0, \ \ (t,x)\in [t_{i-1},t_{i})\times \mathbb{R}^d,
\end{equation*}
with terminal conditions
\begin{equation*}
u_{i}(t_{i},x;x_{1},\cdots,x_{i-1})=u_{i+1}(t_{i},x;x_{1},\cdots, x_{i-1},x),
\, \, \hbox{for $i<k$},
\end{equation*}
and $u_{k}(t_{k},x;x_{1},\cdots,x_{k-1})=\varphi (x_{1},\cdots, x_{k-1},x)$.
The $G$-expectation of $\xi$ is defined by $\mathbb{\hat{E}}[\xi]=%
\mathbb{\hat{E}}_{0}[\xi]$.

For each
$p\geq1$, the completion of $L_{ip}(\Omega)$ under the norm
$||X||_{L_{G}^{p}}:=\left(\mathbb{\hat{E}}[|X|^{p}]\right)^{1/p}$ is denoted by
$L_{G}^{p}(\Omega)$. Similarly, we can define
$L_{G}^{p}(\Omega_{T})$ for each fixed $T\geq0$.   In this paper, we always assume
 that $G$ is non-degenerate to ensure the well-posedness of $G$-BSDE (see appendix A), i.e., there exist two constants
$0<\underline{\sigma}^{2}\leq\overline{\sigma}^{2}<\infty$ such that
\[
\frac{1}{2}\underline{\sigma}^{2}\mathrm{tr}[A-B]\leq G(A)-G(B)\leq \frac{1}{2}\overline{\sigma}^{2}\mathrm{tr}[A-B]\text{ for }A\geq B.
\]

\begin{theorem}[\cite{DHP11,HP09}]
  There exists a weakly compact set
$\mathcal{P}$ of probability
measures on $(\Omega,\mathcal{B}(\Omega))$, such that
\[
\mathbb{\hat{E}}[\xi]=\sup_{P\in\mathcal{P}}E_{P}[\xi]\  \text{for
 all}\ \xi\in  {L}_{G}^{1}{(\Omega)}.
\]
\end{theorem}

Now, we define capacity%
\[
c(A):=\sup_{P\in\mathcal{P}}P(A),\ A\in\mathcal{B}(\Omega).
\]
A set $A\in\mathcal{B}(\Omega)$ is polar if $c(A)=0$.  A
property holds quasi-surely (q.s.) if it holds outside a
polar set. In what follows, we do not distinguish between two random
variables $X$ and $Y$ if $X=Y$ q.s.

\begin{definition}
\label{def2.6} Let $M_{G}^{0}(0,T)$ be the collection of processes
of  the following form: for a given partition
$\{t_{0},\cdot\cdot\cdot,t_{N}\}$ of $[0,T]$,
\[
\eta_{t}(\omega)=\sum_{i=0}^{N-1}\xi_{i}(\omega)I_{[t_{i},t_{i+1})}(t),
\]
where $\xi_{i}\in L_{ip}(\Omega_{t_{i}})$,
$i=0,1,2,\cdot\cdot\cdot,N-1$. For each $p\geq1$,  denote by
$M_{G}^{p}(0,T)$ the completion of $M_{G}^{0}(0,T)$ under the norm
$||\eta||_{M_{G}^{p}}:=\left(\mathbb{\hat{E}}[\int_{0}^{T}|\eta_{t}|^{p}dt]\right)^{1/p}$.
\end{definition}

For simplicity, we  denote by $\langle
B\rangle:=(\langle B^i, B^j\rangle)_{i,j=1}^d$   the cross-variation process.
Denote by $M^{p}_{G}%
	(0,T;\mathbb{R}^{d})$ the set of $d$-dimensional stochastic  process
	$\eta=(\eta^{1},\cdots,\eta^{d})$ such that $\eta^{i}\in M^{p}_{G}(0,T),i\leq d$. Similarly, we can define $L^p_G(\Omega;\mathbb{R}^d)$.
Then, for any $ \eta\in M_{G}^{2}(0,T;\mathbb{R}^d)$ and $ \gamma\in M_{G}^{1}(0,T;\mathbb{S}(d))$,
the $G$-It\^{o} integrals  \[
\int_0^T\eta_sdB_s:=\sum_{i=1}^d\int^T_0\eta^i(s)dB^i_s \ \text{and} \ \int^T_0\gamma_sd\langle
B\rangle_s:=\sum_{i,j=1}^d\int_{0}^{T}\gamma_{ij}({s})d\langle B^i, B^j\rangle_s
\]
 are well defined, see  Peng \cite{P08a,P10}.
Moreover, we also have the corresponding $G$-It\^{o}'s calculus theory. The following inequalities will be used frequently in this paper.

\begin{lemma}[\cite{P08a,P10}]
\label{myw901} Assume that $\gamma \in {M}^{p}_G(0,T;\mathbb{R}^d)$ and $\eta \in {M}^{p}_G(0,T;\mathbb{S}(d))$.  Then, for each $p\geq 2$, we have%
\begin{align*}
 &\mathbb{\hat{E}}\left[\sup_{t\in
\lbrack0,T]}\left|\int_{0}^{t}\eta_{s}dB_{s}\right|^{p}\right]\leq C(p)
\mathbb{\hat{E}}\left[\left(\int_{0}^{T}|\eta_{s}|^{2}ds\right)^{p/2}\right],\\
 &\mathbb{\hat{E}}\left[\left|\int_{0}^{T}\gamma_{s}d\langle B\rangle_s\right|^{p}\right]\leq  C(p,T)
\mathbb{\hat{E}}\left[\int_{0}^{T}|\gamma_{s}|^{p}ds\right].
\end{align*}
\end{lemma}

\subsection{$G$-SDE with two time-scales}
In this section, we shall state some basic results about $G$-SDE \eqref{myw201}, which will be used in our subsequent discussions. Throughout this paper, each element $x\in\mathbb{R}^{2n}$ is identified to  $(\tilde{x},\bar{x})\in\mathbb{R}^n\times\mathbb{R}^n,$ unless otherwise specified.
We need the following assumption:
\begin{description}
\item[(H1)] There exists a constant $L_1>0$ such that, for any $x,x^{\prime}\in\mathbb{R}^{2n}$,
\begin{align*}
&|{\ell}(x)-{\ell}(x^{\prime})|\leq
L_1|x-x^{\prime}|\ \text{and}\ |\ell(0)|\leq L_1,\ \text{for $\ell=\widetilde{b}, \overline{b}, \widetilde{h}_{ij}, \overline{h}_{ij}$, $\widetilde{\sigma}$ and $\overline{\sigma}$.}
\end{align*}
\end{description}

Under assumption {(H1)}, the $G$-SDE \eqref{myw201} has a unique solution $(\widetilde{X}^{\varepsilon,x},\overline{X}^{\varepsilon,x})\in M^2_G(0,T;\mathbb{R}^{2n})$ for each $T>0$ and we refer the reader to
 Chapter V in Peng \cite{P10} or Gao \cite{G} for the proof. Then, for any $\varphi\in C(\mathbb{R}^n)$ of polynomial growth, we define the function
\[
u^{\varepsilon}(t,\tilde{x},\bar{x}):=\mathbb{\hat{E}}\left[\varphi(\widetilde{X}^{\varepsilon,x}_t)\right], \ \forall  x=(\tilde{x},\bar{x})\in\mathbb{R}^n\times\mathbb{R}^{n}.
\]
For convenience, set
\[
b^{\varepsilon}=\left[
\begin{array}
[c]{cc}%
\widetilde{b}\\
\frac{\overline{b}}{\varepsilon}
\end{array}
\right],\ h_{ij}^{\varepsilon}=\left[
\begin{array}
[c]{cc}%
\widetilde{h}_{ij}\\
\frac{\overline{h}_{ij}}{\varepsilon}
\end{array}
\right], \
\sigma^{\varepsilon}=\left[
\begin{array}
[c]{cc}%
\widetilde{\sigma}\\
\frac{\overline{\sigma}}{\sqrt{\varepsilon}}
\end{array}
\right].
\]

Then we have the following result.
\begin{lemma}\label{myw203}
Suppose assumption \emph{(H1)} holds. Then for each $T>0$, $u^{\varepsilon}$ is the unique viscosity solution of the following fully nonlinear PDEs:%
\begin{equation}
\left \{
\begin{array}
[c]{l}%
\partial_{t}u^{\varepsilon}-G\left((\sigma^{\varepsilon})^{\top}D^{2}u^{\varepsilon}\sigma+2[\langle h^{\varepsilon}_{ij},Du^{\varepsilon}\rangle]_{i,j=1}^d\right)-\langle b^{\varepsilon},Du^{\varepsilon}\rangle=0,\ (t,\tilde{x},\bar{x})\in(0,T)\times\mathbb{R}^n\times\mathbb{R}^{n},\\
u^{\varepsilon}(0,\tilde{x},\bar{x})=\varphi(\tilde{x}),\ (\tilde{x},\bar{x})\in\mathbb{R}^n\times\mathbb{R}^n,
\end{array}
\right.  \label{PDE}%
\end{equation}
where $Du^{\varepsilon}=(\partial_{x_i}u^{\varepsilon})_{i=1}^{2n}$ and $D^2u^{\varepsilon}=[\partial_{x_ix_j}^2u^{\varepsilon}]_{i,j=1}^{2n}$ for
each $x=(\tilde{x},\bar{x})\in\mathbb{R}^{2n}$. Moreover, it holds that
\begin{align}\label{myw607}
u^{\varepsilon}(t,\tilde{x},\bar{x})=\mathbb{\hat{E}}\left[u^{\varepsilon}(t-\delta,\widetilde{X}^{\varepsilon,x}_{\delta},\overline{X}^{\varepsilon,x}_{\delta})\right], \ \forall 0\leq \delta\leq t\leq T.
\end{align}
\end{lemma}
\begin{proof}
The proof follows from Theorem 5.3.7  of \cite{P10} or Theorem 4.5 in \cite{HJPS1} and the fact that the $G$-SDE \eqref{myw201} is time-homogeneous.
\end{proof}

From Lemma \ref{myw203}, we could study the asymptotic behavior  of $u^{\varepsilon}$ as $\varepsilon\rightarrow 0$ through the  slow component $\widetilde{X}^{\varepsilon,x}_t$.
In the rest of the article, we are going to discuss the  limit distribution of $\widetilde{X}^{\varepsilon,x}$ as $\varepsilon\rightarrow 0$.

\begin{remark}{\upshape
The equation \eqref{PDE} is a fully nonlinear PDE  without periodic structure, which is
different from the existing research; see \cite{A1, BE1, BE, BI1, Kh, Kh2, PS} and the references therein.
}\end{remark}

\section{The  averaging principle}
This section is devoted to the research of limit behaviour of the slow $G$-diffusion process
as $\varepsilon\rightarrow 0$. In order to describe the  averaged PDE, we introduce the following  auxiliary $G$-SDE: for any $x=(\tilde{x},\bar{x})\in\mathbb{R}^{2n}$,
\begin{align}\label{myw501}
 \overline{X}^{x}_t=\bar{x}+\int^t_0\overline{b}(\tilde{x},\overline{X}^{{x}}_s)ds+\sum\limits_{i,j=1}^d\int^t_0\overline{h}_{ij}(\tilde{x},\overline{X}^{{x}}_s)d\langle B^i,B^j\rangle_s+\int^t_0\overline{\sigma}(\tilde{x},\overline{X}^{{x}}_s)dB_s.
\end{align}

In what follows, we make use of the following assumptions.
\begin{description}
\item[(H2)] There exists  a constant $\eta>0$ such that, for each $\tilde{x},\bar{x},\bar{x}^{\prime}\in\mathbb{R}^n$.\begin{align*}
&G\left((\overline{\sigma}(\tilde{x},\bar{x})-\overline{\sigma}(\tilde{x},\bar{x}^{\prime}%
))^{\top}(\overline{\sigma}(\tilde{x},\bar{x})-\overline{\sigma}(\tilde{x},\bar{x}^{\prime}))+2\left[\langle \bar{x}-\bar{x}^{\prime}%
,\overline{h}_{ij}(\tilde{x},\bar{x})-\overline{h}_{ij}(\tilde{x},\bar{x}^{\prime})\rangle\right]_{i,j=1}^{d}\right)\\&
+\langle \bar{x}-\bar{x}^{\prime
},\overline{b}(\tilde{x},\bar{x})-\overline{b}(\tilde{x},\bar{x}^{\prime})\rangle \leq-\eta|\bar{x}-\bar{x}^{\prime}|^{2}
\end{align*}
\item [(H3)] There exists  a constant $L_2>0$ such that $|\ell(x)|\leq {L}_2(1+|\tilde{x}|)$ for  $\ell=\widetilde{b},\widetilde{h}_{ij}$, $\widetilde{\sigma}$ and  $x=(\tilde{x},\bar{x})\in\mathbb{R}^{2n}$.
\end{description}
\begin{remark}{\upshape
The assumption (H2) is called  dissipativity
condition, which ensures the ergodicity of the diffusion process $\overline{X}^x$ (cf. \cite{DZ1,DH,FH}). The assumption (H3)  is equivalent to  $|\ell(0,\bar{x})|\leq {L}_2$,  which is used to establish a uniform moment estimate of order $p>2$ for the  slow component $\widetilde{X}^{\varepsilon,x}$ (see Remark \ref{myw9019} in section 4).
}
\end{remark}
\begin{lemma}\label{myw123}
Suppose  assumptions \emph{(H1)}-\emph{(H3)} are satisfied. Then, for each $(\tilde{x},\bar{x},p,A)\in\mathbb{R}^{2n}\times\mathbb{R}^n\times\mathbb{S}(n)$,
the following  limit
\begin{align*}
\widetilde{G}(\tilde{x},p,A):= \lim\limits_{t\rightarrow\infty}\frac{1}{t}\mathbb{\hat{E}}\left[\int^t_0\langle p, \widetilde{b}(\tilde{x},\overline{X}^{x}_s)\rangle ds+\sum_{i,j=1}^d\int^t_0 \left(\langle p,\widetilde{h}_{ij}(\tilde{x},\overline{X}^{x}_s)\rangle+\frac{1}{2}\widetilde{\sigma}^A_{ij}(\tilde{x},\overline{X}^{x}_s)\right)d\langle B^i, B^j\rangle_s\right]
\end{align*}
exists and is independent of the argument $\bar{x}$,
where the  matrix $\widetilde{\sigma}^A=[\widetilde{\sigma}^A_{ij}]_{i,j}=\widetilde{\sigma}^{\top}A\widetilde{\sigma}$.
\end{lemma}
\begin{proof}
For each $\bar{x}\in\mathbb{R}^n$, consider the following  ergodic $G$-BSDE: $\forall 0\leq t\leq r<\infty$,
\begin{align*}
{Y}_{t}^{p,A,\bar{x}}   =&{Y}^{p,A,\bar{x}}_r+\int^r_t\left(\langle p, \widetilde{b}(\tilde{x},\overline{X}^{x}_s)\rangle-\widetilde{G}(\tilde{x},p,A)\right)ds+\sum_{i,j=1}^d\int^r_t \left(\langle p,\widetilde{h}_{ij}(\tilde{x},\overline{X}^{x}_s)\rangle+\frac{1}{2}\widetilde{\sigma}^A_{ij}(\tilde{x},\overline{X}^{x}_s)\right)d\langle B^i, B^j\rangle_s\\
& -\int_{t}^{r}{Z}_{s}^{p,A,\bar{x}}dB_{s}-({K}_{r}^{p,A,\bar{x}}-{K}_{t}^{p,A,\bar{x}}).
\end{align*}
Under assumptions (H1) and (H3), it holds that $|\ell(\tilde{x},0)|\leq {L}_1(1+|\tilde{x}|)$  for $\ell=\overline{b},\overline{h}_{ij}$, $\overline{\sigma}$, and
\begin{align}\label{myw122}
\begin{split}
&|\langle p, \widetilde{b}(\tilde{x},\bar{x})\rangle-\langle p, \widetilde{b}(\tilde{x},\bar{x}^{\prime})\rangle|+\sum_{i,j=1}^d \left|\langle p,\widetilde{h}_{ij}(\tilde{x},\bar{x})\rangle-\langle p,\widetilde{h}_{ij}(\tilde{x},\bar{x}^{\prime})\rangle\right|+ \frac{1}{2}\sum_{i,j=1}^d\left|\widetilde{\sigma}^{A}_{ij}(\tilde{x},\bar{x})-\widetilde{\sigma}^{A}_{ij}(\tilde{x},\bar{x}^{\prime})\right|\\
&\leq C(L_1,L_2)(1+|\tilde{x}|)(|p|+|A|)|\bar{x}-\bar{x}^{\prime}|.
\end{split}
\end{align}
Thus, by Lemma \ref{myw6} in appendix A, the above ergodic $G$-BSDE has a  solution \[
\left({Y}^{p,A,\bar{x}},{Z}^{p,A,\bar{x}},{K}^{p,A,\bar{x}},\widetilde{G}(\tilde{x},p,A)\right)\in \mathfrak{S}_{G}^{2}(0,\infty)\times\mathbb{R}.\] Moreover, from Lemma \ref{myw7} in appendix A (taking $\kappa_1=L_1,\kappa_2=C(L_1,L_2)(1+|\tilde{x}|)(|p|+|A|)$ and $\bar{\kappa}={L}_1(1+|\tilde{x}|)$), we have for each $t\in[0,\infty)$
\begin{align}\label{myw125}
\begin{split}
&\left|\mathbb{\hat{E}}\left[{\int^{t}_0\langle p,  \widetilde{b}(\tilde{x},\overline{X}^{{x}}_s)\rangle ds}+\sum_{i,j=1}^d\int^{t}_0\left(\langle p,\widetilde{h}_{ij}(\tilde{x},\overline{X}^{{x}}_s)\rangle+\frac{1}{2}\widetilde{\sigma}^A_{ij}(\tilde{x},\overline{X}^{{x}}_s)\right)d\langle B^i, B^j\rangle_s\right]-\widetilde{G}(\tilde{x},p,A) t\right|\\
&\leq { C(L_1,\eta)C(L_1,L_2)(|p|+|A|)(1+|\tilde{x}|)\left(1+|\bar{x}|+{L}_1(1+|\tilde{x}|)\right)}\\
&\leq { C(L_1,L_2,\eta)(|p|+|A|)(1+|x|^2)},\end{split}
\end{align}
which ends the proof.
\end{proof}

 Moreover,  $\widetilde{G}(\tilde{x},p,A)$ has the following properties.

\begin{lemma}\label{myw502}
Assume the conditions \emph{(H1)}-\emph{(H3)}  hold. Then for each $\tilde{x},\tilde{x}^{\prime},p,p^{\prime}\in\mathbb{R}^n$ and $A,A^{\prime}\in\mathbb{S}(n)$,
\begin{description}
\item[(i)] $\widetilde{G}(\tilde{x},p+{p}^{\prime},A+{A}^{\prime})  \leq \widetilde{G}(\tilde{x},p,A)+\widetilde{G}(\tilde{x},{p}^{\prime},{A}^{\prime})$,
\item[(ii)] $\widetilde{G}(\tilde{x},\lambda p,\lambda A)  =\lambda \widetilde{G}(\tilde{x},p,A)$ for each $ \lambda \geq0$,
\item[(iii)] $\widetilde{G}(\tilde{x},p,A)  \geq \widetilde{G}(\tilde{x},p,{A}^{\prime})$, {if} $ A\geq {A}^{\prime},$
\item[(iv)] $|\widetilde{G}(\tilde{x},p,A)-\widetilde{G}(\tilde{x}^{\prime},p^{\prime},A^{\prime})|\leq C(L_1,L_2,\eta)(1+|\tilde{x}|^2+|\tilde{x}^{\prime}|^2)\left[(|p|+|A|)|\tilde{x}-\tilde{x}^{\prime}|+|{p}-{p}^{\prime}|+|{A}-{A}^{\prime}|\right].$ In particular, $\widetilde{G}$ is a continuous function.
\end{description}
\end{lemma}
\begin{proof}
We only prove Assertion (iv), since the others  are obvious due to the sublinearity of $\mathbb{\hat{E}}$. Without loss of generality, assume that $\widetilde{h}_{ij}=0$,  $i,j=1,\ldots,d$.
 Recalling Assertion (iii) of Lemma 4.3 in \cite{HW2} (taking $\tilde{x}$ as the control argument),  we obtain that
\begin{align}\label{myw503}
\sup\limits_{t\geq 0}\mathbb{\hat{E}}\left[\left|  \overline{X}^{(\tilde{x},\bar{x})}_t- \overline{X}^{(\tilde{x}^{\prime},\bar{x})}_t \right|^2\right]\leq C(L_1,\eta)|\tilde{x}-\tilde{x}^{\prime}|^2, \ \forall \bar{x}\in\mathbb{R}^n.
\end{align}
It follows that
\begin{align*}
\mathbb{\hat{E}}\left[\left|\langle p, \widetilde{b}(\tilde{x},\overline{X}^{(\tilde{x},\bar{x})}_s)\rangle-\langle p^{\prime}, \widetilde{b}(\tilde{x}^{\prime},\overline{X}^{(\tilde{x}^{\prime},\bar{x})}_s)\rangle \right|\right]&\leq |p|\mathbb{\hat{E}}\left[\left| \widetilde{b}(\tilde{x},\overline{X}^{(\tilde{x},\bar{x})}_s)- \widetilde{b}(\tilde{x}^{\prime},\overline{X}^{(\tilde{x}^{\prime},\bar{x})}_s)\right|\right]+|p-p^{\prime}|\mathbb{\hat{E}}\left[\left| \widetilde{b}(\tilde{x}^{\prime},\overline{X}^{(\tilde{x}^{\prime},\bar{x})}_s)\right|\right]\\
&\leq C(L_1,\eta)|p||\tilde{x}-\tilde{x}^{\prime}|+L_2(1+|\tilde{x}^{\prime}|)|p-p^{\prime}|,
\end{align*}
where we have used the fact that
$
\left| \widetilde{b}(\tilde{x}^{\prime},\overline{X}^{(\tilde{x}^{\prime},\bar{x})}_s)\right|\leq L_2(1+|\tilde{x}^{\prime}|)
$
(see assumption (H3)) in the last inequality.
By a similar analysis, we also deduce that
\begin{align*}
&\mathbb{\hat{E}}\left[\left|\widetilde{\sigma}^{\top}(\tilde{x},\overline{X}^{(\tilde{x},\bar{x})}_s)A\widetilde{\sigma}(\tilde{x},\overline{X}^{(\tilde{x},\bar{x})}_s)-\widetilde{\sigma}^{\top}(\tilde{x}^{\prime},\overline{X}^{(\tilde{x}^{\prime},\bar{x})}_s)A^{\prime}\widetilde{\sigma}(\tilde{x}^{\prime},\overline{X}^{(\tilde{x}^{\prime},\bar{x})}_s)\right|\right]\\
&\leq C(L_1,L_2,\eta)(1+|\tilde{x}|+|\tilde{x}^{\prime}|)|A||\tilde{x}-\tilde{x}^{\prime}|+C(L_2)(1+|\tilde{x}^{\prime}|^2)|{A}-{A}^{\prime}|.
\end{align*}
Consequently, by the definition of $\widetilde{G}$ and Lemma \ref{myw901}, we derive that
\begin{align*}
&|\widetilde{G}(\tilde{x},p,A)-\widetilde{G}(\tilde{x}^{\prime},p^{\prime},A^{\prime})|
\leq \limsup_{T\rightarrow\infty} \frac{1}{T}\int^T_0\mathbb{\hat{E}}\left[\left|\langle p, \widetilde{b}(\tilde{x},\overline{X}^{(\tilde{x},\bar{x})}_s)\rangle-\langle p^{\prime}, \widetilde{b}(\tilde{x}^{\prime},\overline{X}^{(\tilde{x}^{\prime},\bar{x})}_s)\rangle \right|\right]ds
\\&\ \ \ \ \ \ +\limsup_{T\rightarrow\infty} \frac{1}{T}\mathbb{\hat{E}}\left[\left|\int^T_0 \left(\widetilde{\sigma}^{\top}(\tilde{x},\overline{X}^{(\tilde{x},\bar{x})}_s)A\widetilde{\sigma}(\tilde{x},\overline{X}^{(\tilde{x},\bar{x})}_s)-\widetilde{\sigma}^{\top}(\tilde{x}^{\prime},\overline{X}^{(\tilde{x}^{\prime},\bar{x})}_s)A^{\prime}\widetilde{\sigma}(\tilde{x}^{\prime},\overline{X}^{(\tilde{x}^{\prime},\bar{x})}_s)\right)d\langle B\rangle_s\right|\right]
\\&\ \ \ \leq C(L_1,L_2,\eta)(1+|\tilde{x}|^2+|\tilde{x}^{\prime}|^2)\left[(|p|+|A|)|\tilde{x}-\tilde{x}^{\prime}|+|{p}-{p}^{\prime}|+|{A}-{A}^{\prime}|\right],
\end{align*}
which   is the desired result.
\end{proof}

Next,  we  introduce the averaged PDE:
\begin{equation}
\begin{cases}
&\partial_{t}\widetilde{u}-\widetilde{G}\left(\tilde{x},D\widetilde{u},D^{2}\widetilde{u}\right)=0,\ \forall (t,\tilde{x})\in(0,T)\times\mathbb{R}^{n},\\
&\widetilde{u}(0,\tilde{x})=\varphi(\tilde{x}),\ \forall\tilde{x}\in\mathbb{R}^n,
 \label{PDE511}%
\end{cases}
\end{equation}
where  $\varphi\in C(\mathbb{R}^n)$ satisfies the polynomial growth condition.
The above PDE has  a unique viscosity solution $\widetilde{u}$ of polynomial growth (see  Theorem \ref{myw521}). For the definition and basic properties of viscosity solution, we
refer the reader to Crandall,  {Ishii} and Lions \cite{CMI}.

\begin{example}\label{myw1020}{
\upshape
Assume that $G(A)=\frac{1}{2}\mathrm{tr}[A]$ and $\widetilde{h}_{ij}=\overline{h}_{ij}=0, i,j=1,\ldots,d$.
Then, the $G$-Brownian motion reduces to a Brownian motion. Denote by
\[
\widetilde{b}(\tilde{x}):=\lim\limits_{t\rightarrow\infty}\frac{1}{t}\mathbb{\hat{E}}\left[\int^t_0\widetilde{b}(\tilde{x},\overline{X}^{x}_s) ds\right],\ \widetilde{a}(\tilde{x}):=\lim\limits_{t\rightarrow\infty}\frac{1}{t}\mathbb{\hat{E}}\left[\int^t_0\widetilde{\sigma}(\tilde{x},\overline{X}^{x}_s)\widetilde{\sigma}^{\top}(\tilde{x},\overline{X}^{x}_s)ds\right], \ \forall \bar{x}\in\mathbb{R}^n.
\]
In this case, the corresponding generator function $\widetilde{G}$ is given by
\[
\widetilde{G}(\tilde{x},p,A)=\langle p,\widetilde{b}(\tilde{x})\rangle+\frac{1}{2}\mathrm{tr}[\widetilde{a}(\tilde{x}) A], \ \forall(\tilde{x},p,A)\in\mathbb{R}^n\times\mathbb{R}^n\times\mathbb{S}(n).
\]
Under some appropriate conditions, Khasminskii \cite{Kh} proved that $u^{\varepsilon}(t,\tilde{x},\bar{x})$ converges to $\widetilde{u}(t,\tilde{x})$ through
 the martingale problem approach. Moreover, $\widetilde{X}^{\varepsilon,x}_t$ converges in law to $\widetilde{X}^{\tilde{x}}_t$, where
\[
\widetilde{X}^{\tilde{x}}_t=\tilde{x}+\int^t_0\widetilde{b}(\widetilde{X}^{\tilde{x}}_s)ds+\int^t_0\sqrt{\widetilde{a}}(\widetilde{X}^{\tilde{x}}_s)dW_s.
\]
Here $\sqrt{\widetilde{a}}$ is a square root of the $n\times n$ matrix $\widetilde{a}$ and $W$ is a $n$-dimensional Brownian motion.
}
\end{example}
\begin{example}{
\upshape
Suppose that all the coefficients of $G$-SDE  are independent of the slow component $\widetilde{X}^{\varepsilon,x}$, and $\widetilde{h}_{ij}=\overline{h}_{ij}=0, i,j=1,\ldots,d$.
Then, the $G$-SDE \eqref{myw201} reduces to
\begin{align}\label{myw208}
\widetilde{X}^{\varepsilon,x}_t=\tilde{x}+\int^t_0\widetilde{b}(\overline{X}^{\varepsilon,{x}}_s)ds+\int^t_0\widetilde{\sigma}(\overline{X}^{\varepsilon,{x}}_s)dB_s,\ \
\overline{X}^{\varepsilon,{x}}_t=\bar{x}+\int^t_0\frac{\overline{b}(\overline{X}^{\varepsilon,{x}}_s)}{\varepsilon}ds+\int^t_0\frac{\overline{\sigma}(\overline{X}^{\varepsilon,{x}}_s)}{\sqrt{\varepsilon}}dB_s.
\end{align}
Furthermore, assume that $\overline{b}(0)=0$  and $\overline{\sigma}(0)=0$. It is obvious that $\overline{X}^{\varepsilon,0}_t=\overline{X}^{0}_t=0$ for each $\varepsilon\in(0,1)$. Thus, from Assertion (ii) of Lemma \ref{myw2} in appendix A, we obtain that
\[\hat{\mathbb{E}}\left[\left|\overline{X}^{\varepsilon,{x}}_t\right|^2\right]\leq \exp\left(-\frac{2\eta t}{\varepsilon}\right)|\bar{x}|^2,\ \forall t\geq 0.\]
It follows  that
\[
\mathbb{\hat{E}}\left[\left|\widetilde{X}^{\varepsilon,x}_t-\tilde{x}-\widetilde{b}(0)t-\widetilde{\sigma}(0)B_t\right| \right]\leq C(L_1,t)\left(\mathbb{\hat{E}}\left[\int^t_0\left|\overline{X}^{\varepsilon,{x}}_s\right|^2ds \right]\right)^{\frac{1}{2}} \leq C(L_1,\eta,t)|\bar{x}|\sqrt{\varepsilon},
\]
which implies that $\widetilde{X}^{\varepsilon,x}_t$ converges  to $\tilde{x}+\widetilde{b}(0)t+\widetilde{\sigma}(0)B_t$ in $L^1_G$-norm.

According to Lemma \ref{myw203}, we can derive that the function \[\widetilde{u}(t,\tilde{x})=\mathbb{\hat{E}}\left[\varphi(\tilde{x}+\widetilde{b}(0)t+\widetilde{\sigma}(0)B_t)\right]\] is the unique viscosity solution to the averaged PDE \eqref{PDE511} with generator
\[
\widetilde{G}(p,A)=\lim\limits_{t\rightarrow\infty}\frac{1}{t}\mathbb{\hat{E}}\left[\int^t_0\langle p, \widetilde{b}(\overline{X}^{0}_s)\rangle ds+\frac{1}{2}\int^T_0\widetilde{\sigma}^A(\overline{X}^{0}_s)d\langle B\rangle_s\right]=\langle p,\widetilde{b}(0)\rangle+G(\widetilde{\sigma}^A(0)).
\]

}
\end{example}

Now, we are in a position to state the main results.
\begin{theorem} \label{myw521}
Suppose assumptions \emph{(H1)}-\emph{(H3)} hold. Then, for each $\varphi\in C(\mathbb{R}^n)$ of  polynomial growth,
the averaged PDE \eqref{PDE511} admits a unique viscosity solution $\widetilde{u}$  satisfying the polynomial growth condition, and
\[
\lim\limits_{\varepsilon\rightarrow0}u^{\varepsilon}(t,\tilde{x},\bar{x})=\widetilde{u}(t,\tilde{x}), \ \forall (t,\tilde{x},\bar{x})\in [0,\infty)\times\mathbb{R}^{2n}.\]
\end{theorem}

The proof of Theorem \ref{myw521} will be given in section 4.  The following result is a direct consequence of Theorem \ref{myw521}.
\begin{corollary}\label{myw1905}
Suppose all the assumptions of Theorem \emph{\ref{myw521}} hold. Then, for each $(t,x)\in [0,\infty)\times\mathbb{R}^{2n}$ with $x=(\tilde{x},\bar{x})$, the slow $G$-diffusion process $\widetilde{X}^{\varepsilon,x}_t$ converges in law as $\varepsilon\rightarrow 0$, i.e.,
\[
\lim\limits_{\varepsilon\rightarrow0}\mathbb{\hat{E}}\left[\varphi(\widetilde{X}^{\varepsilon,x}_t)\right]=\widetilde{u}(t,\tilde{x}).
\]
\end{corollary}

The Corollary \ref{myw1905} indicates that the distribution of the slow component can be approximated by the solution to the averaged PDE \eqref{PDE511}, which is independent of the argument $\bar{x}$.

\begin{example}{
\upshape
Consider the $G$-SDE \eqref{myw208}. Assume that $n=d=1$  and $\widetilde{\sigma}\equiv 0$.
 In this case, the generator $\widetilde{G}$  is given by
\begin{align*}
\widetilde{G}(p)=p^+\lim\limits_{t\rightarrow\infty}\frac{1}{t}\mathbb{\hat{E}}\left[\int^t_0 \widetilde{b}(\overline{X}^{x}_s) ds\right]+p^-\lim\limits_{t\rightarrow\infty}\frac{1}{t}\mathbb{\hat{E}}\left[-\int^t_0 \widetilde{b}(\overline{X}^{x}_s) ds\right]=:\overline{\mu}p^+-\underline{\mu}p^-.
\end{align*}
Then, from Proposition 2.2.7 of \cite{P10}, there exists a maximally distributed random variable
$\widetilde{\zeta}$, such that
the function \[\widetilde{u}(t,\tilde{x}):=\mathbb{\hat{E}}\left[\varphi(\tilde{x}+t\widetilde{\zeta})\right]=\max_{\underline{\mu}\leq r\leq\overline{u}}\varphi(\tilde{x}+rt)\] is the unique viscosity solution to the following PDE:
\begin{equation*}
\begin{cases}
&\partial_{t}\widetilde{u}-\widetilde{G}(D\widetilde{u})=0,\ \forall (t,\tilde{x})\in(0,T)\times\mathbb{R}^{n},\\
&\widetilde{u}(0,\tilde{x})=\varphi(\tilde{x}),\ \forall\tilde{x}\in\mathbb{R}^n.
\end{cases}
\end{equation*}

By Theorem \ref{myw521}, we deduce that $
\widetilde{X}^{\varepsilon,x}_t=\tilde{x}+\int^t_0\widetilde{b}(\overline{X}^{\varepsilon,{x}}_s)ds
$  converges  in law to the maximal distribution $\tilde{x}+t\widetilde{\zeta}$ as $\varepsilon\rightarrow 0$, which can be seen as the law of large number for $G$-diffusion process.
Therefore, we usually  cannot obtain the pointwise convergence of $\widetilde{X}^{\varepsilon,x}_t$ (cf.  \cite{FS}), which is different from the linear case (cf. \cite{FW}).
}
\end{example}

Moreover, with the help of Markov property for $G$-SDEs, we can also deal with  the finite dimensional distribution of the slow $G$-diffusion process.

\begin{theorem}\label{myw1906}
Assume that \emph{(H1)}-\emph{(H3)} hold.
Then, for each $x\in\mathbb{R}^{2n}$ and  $\varphi\in C(\mathbb{R}^{k\times n})$ of  polynomial growth, we have, for any
$0\leq t_1\leq t_2<\cdots t_k<\infty$,
\[
\lim\limits_{\varepsilon\rightarrow0}\mathbb{\hat{E}}\left[\varphi(\widetilde{X}^{\varepsilon,x}_{t_1},\widetilde{X}^{\varepsilon,x}_{t_2},\cdots,\widetilde{X}^{\varepsilon,x}_{t_k})\right]=\lim\limits_{\varepsilon\rightarrow0}\mathbb{\hat{E}}\left[\varphi^{k-1}(\widetilde{X}^{\varepsilon,{x}}_{t_1})\right],
\]
where $\varphi ^{k-1}$ is defined iteratively
through
\begin{eqnarray*}
\varphi ^{1}(\tilde{x}^1,\tilde{x}^2,\cdots \tilde{x}^{k-1}) &=&\lim\limits_{\varepsilon\rightarrow0}\mathbb{\hat{E}}\left[\varphi(\tilde{x}^1,\tilde{x}^2,\cdots, \tilde{x}^{k-1},\widetilde{X}^{\varepsilon,(\tilde{x}^{k-1},0)}_{t_k-t_{k-1}})\right], \\
&&\vdots \\
\varphi ^{k-1}(\tilde{x}^1) &=&\lim\limits_{\varepsilon\rightarrow0}\mathbb{\hat{E}}\left[\varphi^{k-2}(\tilde{x}^1,\widetilde{X}^{\varepsilon,(\tilde{x}^{1},0)}_{t_2-t_{1}})\right].
\end{eqnarray*}%
\end{theorem}
\begin{remark}{\upshape  Applying Theorem \ref{myw1906} and Kolmogorov's criterion for weak compactness (see Lemma \ref{myw6101}) to Example \ref{myw1020}, we can also derive that the slow diffusion process
$\widetilde{X}^{\varepsilon,x}$ converges weakly to  $\widetilde{X}^{\tilde{x}}$, which is the averaging principle for SDEs introduced by \cite{Kh}.
}
\end{remark}

\section{The proof of the main results}
In this section, we shall state the proof of the main results, by making use of nonlinear stochastic calculus and viscosity solution theory, which is different from the linear case.
Roughly speaking, we will prove the limit function  of $u^{\varepsilon}$ is the unique viscosity solution to the averaged PDE \eqref{PDE511}. The key point is based on
the uniform estimate \eqref{myw125}.

First, we  establish a uniform a priori
 estimate of $G$-SDE \eqref{myw201} with two time-scales, which is important for our future discussion. Let $T>0$ be a fixed constant.
\begin{lemma}\label{myw603}
Assume that the conditions  \emph{(H1)} and \emph{(H2)} hold. Then, there exists a constant $C(L_1,\eta,T)$, such that for any $x,x^{\prime}\in\mathbb{R}^{2n}$ and $t\in[0,T]$,
\begin{description}
\item[(i)]$\mathbb{\hat{E}}\left[\sup\limits_{0\leq s\leq T}\left|\widetilde{X}^{\varepsilon,x}_s-\widetilde{X}^{\varepsilon,x^{\prime}}_s\right|^2\right]\leq C(L_1,\eta,T)\left(|\tilde{x}-\tilde{x}^{\prime}|^2+\varepsilon|\bar{x}-\bar{x}^{\prime}|^2\right)$,
\item[(ii)] $\mathbb{\hat{E}}\left[\left|\overline{X}^{\varepsilon,x}_t-\overline{X}^{\varepsilon,x^{\prime}}_t\right|^2\right]\leq C(L_1,\eta,T)\left(|\tilde{x}-\tilde{x}^{\prime}|^2+|\bar{x}-\bar{x}^{\prime}|^2\right)$,
\item[(iii)]$\mathbb{\hat{E}}\left[\sup\limits_{0\leq s\leq T}\left|\widetilde{X}^{\varepsilon,x}_s\right|^2+\left|\overline{X}^{\varepsilon,x}_t\right|^2\right]\leq  C(L_1,\eta,T)\left(1+|\tilde{x}|^2+|\bar{x}|^2\right)$.
\end{description}
\end{lemma}

\begin{proof}
Without loss of generality, assume that $\widetilde{h}_{ij}=\overline{h}_{ij}=0,i,j=1,\ldots,d.$  The proof is divided into the following two steps.

{\bf Step 1} (Assertions {(i)} and {(ii)}).
Applying $G$-It\^{o}'s formula (Proposition 3.6.3  of \cite{P10}) to $e^{\frac{\eta}{\varepsilon}t}\left|\overline{X}^{\varepsilon,x}_t-\overline{X}^{\varepsilon,x^{\prime}}_t\right|^2$ yields that
\begin{align}\label{myw1102}\begin{split}
&e^{\frac{\eta}{\varepsilon}t}\left|\overline{X}^{\varepsilon,x}_t-\overline{X}^{\varepsilon,x^{\prime}}_t\right|^2-|\bar{x}-\bar{x}^{\prime}|^2-M_t\\
&=\frac{\eta}{\varepsilon}\int^t_0e^{\frac{\eta}{\varepsilon}s}\left|\overline{X}^{\varepsilon,x}_s-\overline{X}^{\varepsilon,x^{\prime}}_s\right|^2ds
+\frac{2}{\varepsilon}\int^t_0e^{\frac{\eta}{\varepsilon}s}\left\langle\overline{X}^{\varepsilon,x}_s-\overline{X}^{\varepsilon,x^{\prime}}_s,\overline{b}(\widetilde{X}^{\varepsilon,x}_s,\overline{X}^{\varepsilon,x}_s)-\overline{b}(\widetilde{X}^{\varepsilon,x^{\prime}}_s,\overline{X}^{\varepsilon,x^{\prime}}_s)\right\rangle ds\\ &\ \ +\frac{1}{\varepsilon}\int^t_0e^{\frac{\eta}{\varepsilon}s}\left(\overline{\sigma}(\widetilde{X}^{\varepsilon,x}_s,\overline{X}^{\varepsilon,x}_s)-\overline{\sigma}(\widetilde{X}^{\varepsilon,x^{\prime}}_s,\overline{X}^{\varepsilon,x^{\prime}}_s)\right)^{\top}\left(\overline{\sigma}(\widetilde{X}^{\varepsilon,x}_s,\overline{X}^{\varepsilon,x}_s)-\overline{\sigma}(\widetilde{X}^{\varepsilon,x^{\prime}}_s,\overline{X}^{\varepsilon,x^{\prime}}_s)\right)d\langle B\rangle_s,
\end{split}
\end{align}
where  $M_t:=\frac{2}{\sqrt{\varepsilon}}\int^t_0e^{\frac{\eta}{\varepsilon}s}\left(\overline{X}^{\varepsilon,x}_s-\overline{X}^{\varepsilon,x^{\prime}}_s\right)\left(\overline{\sigma}(\widetilde{X}^{\varepsilon,x}_s,\overline{X}^{\varepsilon,x}_s)-\overline{\sigma}(\widetilde{X}^{\varepsilon,x^{\prime}}_s,\overline{X}^{\varepsilon,x^{\prime}}_s)\right)dB_s$  is a symmetric $G$-martingale with $M_0=0$, i.e., $\mathbb{\hat{E}}[M_t]=-\mathbb{\hat{E}}[-M_t]=0.$

Recalling assumption (H1), we deduce that
\begin{align*}
&\left\langle\overline{X}^{\varepsilon,x}_s-\overline{X}^{\varepsilon,x^{\prime}}_s,\overline{b}(\widetilde{X}^{\varepsilon,x}_s,\overline{X}^{\varepsilon,x}_s)-\overline{b}(\widetilde{X}^{\varepsilon,x^{\prime}}_s,\overline{X}^{\varepsilon,x^{\prime}}_s)\right\rangle\\
&=
\left\langle\overline{X}^{\varepsilon,x}_s-\overline{X}^{\varepsilon,x^{\prime}}_s,\overline{b}(\widetilde{X}^{\varepsilon,x}_s,\overline{X}^{\varepsilon,x}_s)-\overline{b}(\widetilde{X}^{\varepsilon,x}_s,\overline{X}^{\varepsilon,x^{\prime}}_s)+\overline{b}(\widetilde{X}^{\varepsilon,x}_s,\overline{X}^{\varepsilon,x^{\prime}}_s)-\overline{b}(\widetilde{X}^{\varepsilon,x^{\prime}}_s,\overline{X}^{\varepsilon,x^{\prime}}_s)\right\rangle\\
&\leq \left\langle\overline{X}^{\varepsilon,x}_s-\overline{X}^{\varepsilon,x^{\prime}}_s,\overline{b}(\widetilde{X}^{\varepsilon,x}_s,\overline{X}^{\varepsilon,x}_s)-\overline{b}(\widetilde{X}^{\varepsilon,x}_s,\overline{X}^{\varepsilon,x^{\prime}}_s)\right\rangle+L_1\left|\overline{X}^{\varepsilon,x}_s-\overline{X}^{\varepsilon,x^{\prime}}_s\right|\left|\widetilde{X}^{\varepsilon,x}_s-\widetilde{X}^{\varepsilon,x^{\prime}}_s\right|.
\end{align*}
By a similar analysis, we can obtain that
\begin{align*}
&\left(\overline{\sigma}(\widetilde{X}^{\varepsilon,x}_s,\overline{X}^{\varepsilon,x}_s)-\overline{\sigma}(\widetilde{X}^{\varepsilon,x^{\prime}}_s,\overline{X}^{\varepsilon,x^{\prime}}_s)\right)^{\top}\left(\overline{\sigma}(\widetilde{X}^{\varepsilon,x}_s,\overline{X}^{\varepsilon,x}_s)-\overline{\sigma}(\widetilde{X}^{\varepsilon,x^{\prime}}_s,\overline{X}^{\varepsilon,x^{\prime}}_s)\right)\\
&\leq \left(\overline{\sigma}(\widetilde{X}^{\varepsilon,x}_s,\overline{X}^{\varepsilon,x}_s)-\overline{\sigma}(\widetilde{X}^{\varepsilon,x}_s,\overline{X}^{\varepsilon,x^{\prime}}_s)\right)^{\top}\left(\overline{\sigma}(\widetilde{X}^{\varepsilon,x}_s,\overline{X}^{\varepsilon,x}_s)-\overline{\sigma}(\widetilde{X}^{\varepsilon,x}_s,\overline{X}^{\varepsilon,x^{\prime}}_s)\right)\\&\ \ \ \ \ \ \ \ +C(L_1)\left(\left|\overline{X}^{\varepsilon,x}_s-\overline{X}^{\varepsilon,x^{\prime}}_s\right|\left|\widetilde{X}^{\varepsilon,x}_s-\widetilde{X}^{\varepsilon,x^{\prime}}_s\right|+\left|\widetilde{X}^{\varepsilon,x}_s-\widetilde{X}^{\varepsilon,x^{\prime}}_s\right|^2\right)I_{n}.
\end{align*}
In  view of Corollary 3.5.8 of \cite{P10}, we have that,  for each $\eta\in M^1_G(0,T;\mathbb{S}(d))$,
\[\int^t_0\eta_sd\langle B\rangle_s- 2\int^t_0G(\eta_s)ds\leq 0.\]
 Then, according to inequality \eqref{myw1102}, we get that
\begin{align*}
&e^{\frac{\eta}{\varepsilon}t}\left|\overline{X}^{\varepsilon,x}_t-\overline{X}^{\varepsilon,x^{\prime}}_t\right|^2-|\bar{x}-\bar{x}^{\prime}|^2-M_t\\
&\leq \frac{\eta}{\varepsilon}\int^t_0e^{\frac{\eta}{\varepsilon}s}\left|\overline{X}^{\varepsilon,x}_s-\overline{X}^{\varepsilon,x^{\prime}}_s\right|^2ds
+\frac{2}{\varepsilon}\int^t_0e^{\frac{\eta}{\varepsilon}s}\left\langle\overline{X}^{\varepsilon,x}_s-\overline{X}^{\varepsilon,x^{\prime}}_s,\overline{b}(\widetilde{X}^{\varepsilon,x}_s,\overline{X}^{\varepsilon,x}_s)-\overline{b}(\widetilde{X}^{\varepsilon,x}_s,\overline{X}^{\varepsilon,x^{\prime}}_s)\right\rangle ds\\ &\ \ +\frac{2}{\varepsilon}\int^t_0e^{\frac{\eta}{\varepsilon}s}G\left(\left(\overline{\sigma}(\widetilde{X}^{\varepsilon,x}_s,\overline{X}^{\varepsilon,x}_s)-\overline{\sigma}(\widetilde{X}^{\varepsilon,x}_s,\overline{X}^{\varepsilon,x^{\prime}}_s)\right)^{\top}\left(\overline{\sigma}(\widetilde{X}^{\varepsilon,x}_s,\overline{X}^{\varepsilon,x}_s)-\overline{\sigma}(\widetilde{X}^{\varepsilon,x}_s,\overline{X}^{\varepsilon,x^{\prime}}_s)\right)\right)ds\\& \ \
+\frac{C(L_1)}{\varepsilon}\int^t_0e^{\frac{\eta}{\varepsilon}s}\left(\left|\overline{X}^{\varepsilon,x}_s-\overline{X}^{\varepsilon,x^{\prime}}_s\right|\left|\widetilde{X}^{\varepsilon,x}_s-\widetilde{X}^{\varepsilon,x^{\prime}}_s\right|+\left|\widetilde{X}^{\varepsilon,x}_s-\widetilde{X}^{\varepsilon,x^{\prime}}_s\right|^2\right)ds,
\end{align*}
which together with condition (H2)  and \[
\left|\overline{X}^{\varepsilon,x}_t-\overline{X}^{\varepsilon,x^{\prime}}_t\right|\left|\widetilde{X}^{\varepsilon,x}_t-\widetilde{X}^{\varepsilon,x^{\prime}}_t\right|\leq \frac{\eta}{C(L_1)}\left|\overline{X}^{\varepsilon,x}_t-\overline{X}^{\varepsilon,x^{\prime}}_t\right|^2+\frac{C(L_1)}{4\eta}\left|\widetilde{X}^{\varepsilon,x}_t-\widetilde{X}^{\varepsilon,x^{\prime}}_t\right|^2,
\]
implies that
\begin{align*}
e^{\frac{\eta}{\varepsilon}t}\left|\overline{X}^{\varepsilon,x}_t-\overline{X}^{\varepsilon,x^{\prime}}_t\right|^2\leq
|\bar{x}-\bar{x}^{\prime}|^2+M_t+ \frac{C(L_1,\eta)}{\varepsilon}\int^t_0e^{\frac{\eta}{\varepsilon}s}\left|\widetilde{X}^{\varepsilon,x}_s-\widetilde{X}^{\varepsilon,x^{\prime}}_s\right|^2ds.
\end{align*}
Taking $G$-expectation to both sides, we obtain that for each $t\in[0,T]$,
\begin{align}\label{myw600}
\mathbb{\hat{E}}\left[\left|\overline{X}^{\varepsilon,x}_t-\overline{X}^{\varepsilon,x^{\prime}}_t\right|^2\right]\leq
|\bar{x}-\bar{x}^{\prime}|^2e^{-\frac{\eta}{\varepsilon}t}+ \frac{C(L_1,\eta)}{\varepsilon}\int^t_0e^{\frac{\eta}{\varepsilon}(s-t)}\mathbb{\hat{E}}\left[\left|\widetilde{X}^{\varepsilon,x}_s-\widetilde{X}^{\varepsilon,x^{\prime}}_s\right|^2\right]ds.
\end{align}
In particular, it holds that
\begin{align}\label{myw601}
\int^t_0\mathbb{\hat{E}}\left[\left|\overline{X}^{\varepsilon,x}_s-\overline{X}^{\varepsilon,x^{\prime}}_s\right|^2\right]ds
\leq \frac{\varepsilon}{\eta}|\bar{x}-\bar{x}^{\prime}|^2+C(L_1,\eta)\int^t_0\mathbb{\hat{E}}\left[\sup\limits_{0\leq r\leq s }\left|\widetilde{X}^{\varepsilon,x}_r-\widetilde{X}^{\varepsilon,x^{\prime}}_r\right|^2\right]ds.
\end{align}

On the other hand, applying H\"{o}lder's inequality and BDG's inequality, we conclude that
\begin{align*}
\mathbb{\hat{E}}\left[\sup\limits_{0\leq s\leq t}\left|\widetilde{X}^{\varepsilon,x}_s-\widetilde{X}^{\varepsilon,x^{\prime}}_s\right|^2\right]
\leq& 3|\tilde{x}-\tilde{x}^{\prime}|^2+C(L_1,T)\int^t_0\mathbb{\hat{E}}\left[ \left|\overline{X}^{\varepsilon,x}_s-\overline{X}^{\varepsilon,x^{\prime}}_s\right|^2+ \left|\widetilde{X}^{\varepsilon,x}_s-\widetilde{X}^{\varepsilon,x^{\prime}}_s\right|^2  \right]  ds\\
\leq & C(L_1,\eta,T)\left(|\tilde{x}-\tilde{x}^{\prime}|^2+\varepsilon|\bar{x}-\bar{x}^{\prime}|^2+\int^t_0\mathbb{\hat{E}}\left[\sup\limits_{0\leq r\leq s}\left|\widetilde{X}^{\varepsilon,x}_r-\widetilde{X}^{\varepsilon,x^{\prime}}_r\right|^2\right]ds\right),
\end{align*}
where we have used the estimate \eqref{myw601} in the last inequality.
It follows from Gronwall's inequality that
\begin{align*}
\mathbb{\hat{E}}\left[\sup\limits_{0\leq s\leq T}\left|\widetilde{X}^{\varepsilon,x}_s-\widetilde{X}^{\varepsilon,x^{\prime}}_s\right|^2\right]
\leq C(L_1,\eta,T)\left(|\tilde{x}-\tilde{x}^{\prime}|^2+\varepsilon|\bar{x}-\bar{x}^{\prime}|^2\right).
\end{align*}
With the help of  inequality \eqref{myw600}, we obtain that
\begin{align*}
\mathbb{\hat{E}}\left[\left|\overline{X}^{\varepsilon,x}_t-\overline{X}^{\varepsilon,x^{\prime}}_t\right|^2\right]\leq C(L_1,\eta,T)\left(|\tilde{x}-\tilde{x}^{\prime}|^2+|\bar{x}-\bar{x}^{\prime}|^2\right), \ \forall t\in[0,T],
\end{align*}
which is the desired result.

{\bf Step 2} (Assertion {(iii)}).  Applying $G$-It\^{o}'s formula again, we deduce that
\begin{align}\label{myw605}\begin{split}
&e^{\frac{\eta}{\varepsilon}t}\left|\overline{X}^{\varepsilon,x}_t\right|^2-|\bar{x}|^2-M^{\prime}_t-\frac{\eta}{\varepsilon}\int^t_0e^{\frac{\eta}{\varepsilon}s}\left|\overline{X}^{\varepsilon,x}_s\right|^2ds\\
&\leq
\frac{2}{\varepsilon}\int^t_0e^{\frac{\eta}{\varepsilon}s}\left[\left\langle\overline{X}^{\varepsilon,x}_s,\overline{b}(\widetilde{X}^{\varepsilon,x}_s,\overline{X}^{\varepsilon,x}_s)\right\rangle+G\left(\left(\overline{\sigma}(\widetilde{X}^{\varepsilon,x}_s,\overline{X}^{\varepsilon,x}_s)\right)^{\top}\overline{\sigma}(\widetilde{X}^{\varepsilon,x}_s,\overline{X}^{\varepsilon,x}_s)\right)\right]ds,
\end{split}
\end{align}
where  $M^{\prime}_t$  is a symmetric $G$-martingale with $M^{\prime}_0=0$.

In view of assumptions (H1) and (H2), we get that
\begin{align*}
&2\left\langle\overline{X}^{\varepsilon,x}_s,\overline{b}(\widetilde{X}^{\varepsilon,x}_s,\overline{X}^{\varepsilon,x}_s)\right\rangle+2G\left(\left(\overline{\sigma}(\widetilde{X}^{\varepsilon,x}_s,\overline{X}^{\varepsilon,x}_s)\right)^{\top}\left(\overline{\sigma}(\widetilde{X}^{\varepsilon,x}_s,\overline{X}^{\varepsilon,x}_s)\right)\right)\\
&\leq 2\left\langle\overline{X}^{\varepsilon,x}_s,\overline{b}(\widetilde{X}^{\varepsilon,x}_s,\overline{X}^{\varepsilon,x}_s)-\overline{b}(\widetilde{X}^{\varepsilon,x}_s,0)\right\rangle+C(L_1)\left(1+\left|\overline{X}^{\varepsilon,x}_s\right|+\left|\overline{X}^{\varepsilon,x}_s\right|\left|\widetilde{X}^{\varepsilon,x}_s\right|+\left|\widetilde{X}^{\varepsilon,x}_s\right|^2\right)
\\ &\ \ \ \ \ +2G\left(\left(\overline{\sigma}(\widetilde{X}^{\varepsilon,x}_s,\overline{X}^{\varepsilon,x}_s)-\overline{\sigma}(\widetilde{X}^{\varepsilon,x}_s,0)\right)^{\top}\left(\overline{\sigma}(\widetilde{X}^{\varepsilon,x}_s,\overline{X}^{\varepsilon,x}_s)-\overline{\sigma}(\widetilde{X}^{\varepsilon,x}_s,0)\right)\right)\\
&\leq -\eta\left|\overline{X}^{\varepsilon,x}_s\right|^2+C(L_1,\eta)\left(1+\left|\widetilde{X}^{\varepsilon,x}_s\right|^2\right).
\end{align*}
It follows from inequality \eqref{myw605} that\begin{align}\label{myw606}
\mathbb{\hat{E}}\left[\left|\overline{X}^{\varepsilon,x}_t\right|^2\right]\leq e^{-\frac{\eta}{\varepsilon}t}|\bar{x}|^2+\frac{C(L_1,\eta)}{\varepsilon}\int^t_0e^{\frac{\eta}{\varepsilon}(s-t)}\left(1+\mathbb{\hat{E}}\left[\left|\widetilde{X}^{\varepsilon,x}_s\right|^2\right]\right)ds, \ \forall t\in[0,T].
\end{align}

On the other hand, using inequality \eqref{myw606} and by a similar analysis as step 1, we obtain that
\begin{align*}
\mathbb{\hat{E}}\left[\sup\limits_{0\leq s\leq t}\left|\widetilde{X}^{\varepsilon,x}_s\right|^2\right]
\leq& 3|\tilde{x}|^2+C(L_1,T)\int^t_0\mathbb{\hat{E}}\left[ 1+ \left|\overline{X}^{\varepsilon,x}_s\right|^2+ \left|\widetilde{X}^{\varepsilon,x}_s\right|^2  \right]  ds\\
\leq& C(L_1,\eta,T)\left(1+|\tilde{x}|^2+\varepsilon|\bar{x}|^2+\int^t_0\mathbb{\hat{E}}\left[ \sup\limits_{0\leq r\leq s}\left|\widetilde{X}^{\varepsilon,x}_r\right|^2 \right] ds\right).
\end{align*}
Consequently, it holds that
\[
\mathbb{\hat{E}}\left[\sup\limits_{0\leq s\leq T}\left|\widetilde{X}^{\varepsilon,x}_s\right|^2\right]+\sup\limits_{0\leq t\leq T}\mathbb{\hat{E}}\left[\left|\overline{X}^{\varepsilon,x}_t\right|^2\right]\leq  C(L_1,\eta,T)\left(1+|\tilde{x}|^2+|\bar{x}|^2\right).
\]
The proof is complete.
\end{proof}

Then, we have the following asymptotic properties of $u^{\varepsilon}$.

\begin{lemma}\label{myw608} Let $\varphi$ be in $C_{b.lip}(\mathbb{R}^n)$.
Suppose  assumptions \emph{(H1)} and \emph{(H2)} are satisfied.
Then, there exist a sequence $\varepsilon_m\downarrow 0$, $m\geq 1$  and a function $\widetilde{u}^*\in C_{b}([0,T]\times\mathbb{R}^n)$, such that
 for each $s,t\in[0,T]$ and $ x=(\tilde{x},\bar{x})\in\mathbb{R}^{2n}$,
\begin{description}
 \item[(i)]
$
\lim\limits_{m\rightarrow\infty}u^{\varepsilon_m}(t,\tilde{x},\bar{x})=\widetilde{u}^*(t,\tilde{x}),
$
 \item[(ii)] $\lim\limits_{m\rightarrow\infty}\mathbb{\hat{E}}\left[\left|u^{\varepsilon_m}(s,\widetilde{X}^{\varepsilon_m,x}_{t},\overline{X}^{\varepsilon_m,x}_{t})-\widetilde{u}^*(s,\widetilde{X}^{\varepsilon_m,x}_{t})\right|\right]=0$.
\end{description}
\end{lemma}
\begin{proof}
From Lemma \ref{myw603}, we have that, for any  $s,t\in[0,T]$ and $ x,x^{\prime}\in\mathbb{R}^{2n}$,
\begin{align*}
\mathbb{\hat{E}}\left[\left|\widetilde{X}^{\varepsilon,x}_t-\widetilde{X}^{\varepsilon,x^{\prime}}_s\right|^2\right]
\leq& 2\mathbb{\hat{E}}\left[\left|\widetilde{X}^{\varepsilon,x}_t-\widetilde{X}^{\varepsilon,x}_s\right|^2\right]+2\mathbb{\hat{E}}\left[\left|\widetilde{X}^{\varepsilon,x}_s-\widetilde{X}^{\varepsilon,x^{\prime}}_s\right|^2\right]\\
\leq& C(L_1,\eta,T)\left((1+|x|^2)|t-s|+|\tilde{x}-\tilde{x}^{\prime}|^2+\varepsilon|\bar{x}-\bar{x}^{\prime}|^2\right).
\end{align*}
It follows from the definition of $u^{\varepsilon}$ that $|u^{\varepsilon}(t,\tilde{x},\bar{x})|\leq C(\varphi)$ and
\begin{align*}
|u^{\varepsilon}(t,\tilde{x},\bar{x})-u^{\varepsilon}(s,\tilde{x}^{\prime},\bar{x}^{\prime})|&\leq \mathbb{\hat{E}}\left[\left|\varphi(\widetilde{X}^{\varepsilon,x}_t)-\varphi(\widetilde{X}^{\varepsilon,x^{\prime}}_s)\right|\right]\leq C(\varphi)\mathbb{\hat{E}}\left[\left|\widetilde{X}^{\varepsilon,x}_t-\widetilde{X}^{\varepsilon,x^{\prime}}_s\right|\right] \\&
\leq C(L_1,\eta,T,\varphi)\left((1+|x|)\sqrt{|t-s|}+|\tilde{x}-\tilde{x}^{\prime}|+\sqrt{\varepsilon}|\bar{x}-\bar{x}^{\prime}|\right).
\end{align*}
Thus, by the Arzel\`{a}-Ascoli theorem, we can find a  sequence $\varepsilon_m\downarrow 0$, such that
$u^{\varepsilon_m}(t,\tilde{x},\bar{x})$ is a Cauchy sequence for any   $(t,\tilde{x},\bar{x})\in[0,T]\times\mathbb{R}^{2n}$.
Denote
$
\widetilde{u}^*(t,\tilde{x},\bar{x}):=\lim\limits_{m\rightarrow\infty}u^{\varepsilon_m}(t,\tilde{x},\bar{x}).
$
It is obvious that
\[
|\widetilde{u}^*(t,\tilde{x},\bar{x})-\widetilde{u}^*(s,\tilde{x}^{\prime},\bar{x}^{\prime})|\leq C(L_1,\eta,T,\varphi)\left((1+|x|)\sqrt{|t-s|}+|\tilde{x}-\tilde{x}^{\prime}|\right), \ \forall \bar{x},\bar{x}^{\prime}\in\mathbb{R}^n,
\]
which indicates that $\widetilde{u}^*$ is independent of the argument $\bar{x}$.

Next, we will prove Assertion (ii). For each $N>0$, we get that
\begin{align*}
&\mathbb{\hat{E}}\left[\left|u^{\varepsilon_m}(s,\widetilde{X}^{\varepsilon_m,x}_{t},\overline{X}^{\varepsilon_m,x}_{t})-\widetilde{u}^*(s,\widetilde{X}^{\varepsilon_m,x}_{t})\right|\right]\\
&\leq \mathbb{\hat{E}}\left[\left|u^{\varepsilon_m}({s},\widetilde{X}^{\varepsilon_m,x}_{t},\overline{X}^{\varepsilon_m,x}_{t})-\widetilde{u}^*({s},\widetilde{X}^{\varepsilon_m,x}_{t})\right|I_{\left\{\left|\widetilde{X}^{\varepsilon_m,x}_{t}\right|\leq N\right\}}I_{\left\{\left|\overline{X}^{\varepsilon_m,x}_{t}\right|\leq N\right\}}\right] \\
&\ \ \ \ \ \ \ \ \ + C(\varphi)\mathbb{\hat{E}}\left[I_{\left\{\left|\widetilde{X}^{\varepsilon_m,x}_{t}\right|\geq N\right\}}+I_{\left\{\left|\overline{X}^{\varepsilon_m,x}_{t}\right|\geq N\right\}}\right]\\
&\leq \sup\limits_{s\in[0,T],|\tilde{x}|,|\bar{x}|\leq N}|u^{\varepsilon_m}(s,\tilde{x},\bar{x})-\widetilde{u}^*(s,\tilde{x})|+\frac{C(\varphi)}{N}\mathbb{\hat{E}}\left[\left|\widetilde{X}^{\varepsilon_m,x}_{t}\right|+\left|\overline{X}^{\varepsilon_m,x}_{t}\right|\right].
\end{align*}
Note that $u^{\varepsilon_m}$ converges  uniformly to $\widetilde{u}^*$ on every compact subset of $[0,T]\times\mathbb{R}^{2n}.$
Thus, with the help Assertion (iii) of Lemma \ref{myw603}, we conclude that
\[
\limsup\limits_{m\rightarrow \infty}\mathbb{\hat{E}}\left[\left|u^{\varepsilon_m}(s,\widetilde{X}^{\varepsilon_m,x}_{t},\overline{X}^{\varepsilon_m,x}_{t})-\widetilde{u}^*(s,\widetilde{X}^{\varepsilon_m,x}_{t})\right|\right]\leq \frac{ C(L_1,\eta,T,\varphi)}{N}(1+|\tilde{x}|+|\bar{x}|), \ \forall N>0.
\]
Sending $N\rightarrow\infty$ yields the desired result.
\end{proof}

Next, we  show that the function $\widetilde{u}^*$ constructed  above is the viscosity solution to PDE \eqref{PDE511}. For this purpose, we need the following two lemmas.
\begin{lemma}\label{myw6101}
Suppose \emph{(H1)} and \emph{(H3)} hold. Then, for any $p\geq 2$ and $t,s\in[0,T]$,
\begin{align*}&\mathbb{\hat{E}}\left[\sup\limits_{0\leq s\leq T}\left|\widetilde{X}^{\varepsilon,x}_s\right|^p\right]\leq  C(L_1,L_2,p,T)\left(1+|\tilde{x}|^p\right), \\
& \mathbb{\hat{E}}\left[\left|\widetilde{X}^{\varepsilon,x}_t-\widetilde{X}^{\varepsilon,x}_s\right|^p\right]\leq  C(L_1,L_2,p,T)\left(1+|\tilde{x}|^p\right)|t-s|^{\frac{p}{2}}.\end{align*}
\end{lemma}
\begin{proof}
The proof is immediate from BDG's inequality and Gronwall's inequality.
\end{proof}
\begin{lemma}\label{myw619}Suppose \emph{(H1)} holds. Then,
for each   $\rho\in L^1_G(\Omega)$ and  element $\Gamma\in \mathbb{S}(n)$, it holds that
\begin{align*}
\mathbb{\hat{E}}\left[\rho+\left\langle \Gamma\int^{t}_0\widetilde{\sigma}({X}^{\varepsilon,x}_s)dB_s,\int^{t}_0\widetilde{\sigma}({X}^{\varepsilon,x}_s)dB_s\right\rangle-\int^{t}_0\widetilde{\sigma}^{\Gamma}({X}^{\varepsilon,x}_s)d\langle B\rangle_s\right]=\mathbb{\hat{E}}\left[\rho\right], \ \forall t>0,
\end{align*}
where $\widetilde{\sigma}^{\Gamma}$ is given by Lemma \emph{\ref{myw123}}.
\end{lemma}
\begin{proof}
Suppose that $\left\{M_{t}\right\}_{t\geq 0}$ is a symmetric $G$-martingale,
 i.e., $-\mathbb{\hat{E}}\left[-M_{t}\right]=\mathbb{\hat{E}}\left[M_{t}\right]$.
Then, using the property of  $G$-expectation (Proposition 1.3.7  of \cite{P10}),
we get that
\begin{align}\label{myw616}
\mathbb{\hat{E}}\left[\rho+M_{t}\right]=\mathbb{\hat{E}}\left[\rho\right].
\end{align}

On the other hand, applying $G$-It\^{o}'s formula and recalling the definition of $\widetilde{\sigma}^{\Gamma}$, we can get that,
\[
\left\langle \Gamma\int^{t}_0\widetilde{\sigma}({X}^{\varepsilon,x}_s)dB_s,\int^{t}_0\widetilde{\sigma}({X}^{\varepsilon,x}_s)dB_s\right\rangle-\int^{t}_0\widetilde{\sigma}^{\Gamma}({X}^{\varepsilon,x}_s)d\langle B\rangle_s
\]
is  a symmetric $G$-martingale. It follows from equation \eqref{myw616} that the desired result holds.
\end{proof}

\begin{lemma}\label{myw611} Assume all the conditions of Lemma \emph{\ref{myw608}} are satisfied. Furthermore, suppose assumption \emph{(H3)} holds.
Then,  $\widetilde{u}^*$ is the unique viscosity solution to PDE \eqref{PDE511}.
\end{lemma}
\begin{proof} The uniqueness can be obtained by   applying Lemma \ref{myw507} in appendix B.
It suffices to prove that  $\widetilde{u}^*$ is a viscosity subsolution, since the other case can be proved in a similar fashion.
Without loss of generality, assume that $\widetilde{h}_{ij},\overline{h}_{ij}=0,$ $i,j=1,\ldots,d$.

Note that $\widetilde{u}^*$ is a bounded function.  Then, assume that the test function $\psi\in C^{3}_b([0,T]\times\mathbb{R}^n)$ satisfies that $\psi\geq \widetilde{u}^* $ and $\psi(t,\tilde{x})=\widetilde{u}^*(t,\tilde{x})$
for some point $(t,\tilde{x})\in(0,T)\times\mathbb{R}^n$, where $C^{3}_b([0,T]\times\mathbb{R}^n)$
is the space of  the bounded real-valued functions   that are continuously
differentiable up to the third order and whose derivatives of order from $1$ to $3$ are bounded. We need to prove that
\begin{align}\label{myw612}
H(t,\tilde{x},\psi):=\partial_t\psi(t,\tilde{x})-\widetilde{G}(\tilde{x},D\psi(t,\tilde{x}),D^2\psi(t,\tilde{x}))\leq 0.
\end{align}
The proof is divided into the following four steps.

{\bf Step 1} (Dynamic programming principle). Use the same notations as Lemma \ref{myw608}.
Recalling equation \eqref{myw607}, we obtain that, for each $\delta\in (0,1)$ and for any $\bar{x}\in\mathbb{R}^n$
\[
u^{\varepsilon_m}(t,\tilde{x},\bar{x})=\mathbb{\hat{E}}\left[u^{\varepsilon_m}(t-\delta,\widetilde{X}^{\varepsilon_m,x}_{\delta},\overline{X}^{\varepsilon_m,x}_{\delta})\right],
\]
which together with  Assertion (ii) of Lemma \ref{myw608} implies that
\begin{align*}
\widetilde{u}^*(t,\tilde{x})=\lim\limits_{m\rightarrow\infty}\mathbb{\hat{E}}\left[\widetilde{u}^*(t-\delta,\widetilde{X}^{\varepsilon_m,x}_{\delta})\right].
\end{align*}
It follows that
\begin{align}\label{myw615}
\psi(t,\tilde{x})\leq\limsup\limits_{m\rightarrow\infty}\mathbb{\hat{E}}\left[\psi(t-\delta,\widetilde{X}^{\varepsilon_m,x}_{\delta})\right].
\end{align}

{\bf Step 2} (The subsolution property).
For each $m\geq 1$,  define
\[
\xi^{1,m}=\int^{\delta}_{0}\widetilde{b}(\widetilde{X}^{\varepsilon_m,{x}}_s,\overline{X}^{\varepsilon_m,{x}}_s)ds\ \text{and}\ \xi^{2,m}=\int^{\delta}_{0}\widetilde{\sigma}(\widetilde{X}^{\varepsilon_m,{x}}_s,\overline{X}^{\varepsilon_m,{x}}_s)dB_s.
\]

Note that
\begin{align}\label{myw110}\begin{split}
\psi(t-\delta,\widetilde{X}^{\varepsilon_m,x}_{\delta})-\psi(t,\tilde{x})=\psi(t-\delta,\widetilde{X}^{\varepsilon_m,x}_{\delta})-\psi(t,\widetilde{X}^{\varepsilon_m,x}_{\delta})
+\psi(t,\widetilde{X}^{\varepsilon_m,x}_{\delta})-
\psi(t,\tilde{x}).\end{split}
\end{align}
Then, applying Taylor's expansion yields that,
\begin{align*}
&\psi(t-\delta,\widetilde{X}^{\varepsilon_m,x}_{\delta})-\psi(t,\widetilde{X}^{\varepsilon_m,x}_{\delta})
=-\partial_t\psi(t,\tilde{x})\delta+\epsilon^{1,m},\\
&\psi(t,\widetilde{X}^{\varepsilon_m,x}_{\delta})-
\psi(t,\tilde{x})=\left\langle D\psi(t,\tilde{x}),\left(\xi^{1,m}+\xi^{2,m}\right)\right\rangle+\frac{1}{2}\left\langle D^2\psi(t,\tilde{x})\xi^{2,m},\xi^{2,m}\right\rangle
+\epsilon^{2,m},
\end{align*}
with
\begin{align*}
\epsilon^{1,m}&=\delta\int^1_0\left[-\partial_t\psi(t-\alpha\delta,\widetilde{X}^{\varepsilon_m,x}_{\delta})+\partial_t\psi(t,\tilde{x})\right]d\alpha,\\
\epsilon^{2,m}&=\int^1_0 \left\langle D\psi(t,\tilde{x}+\xi^{2,m}+\alpha\xi^{1,m})-D\psi(t,\tilde{x}),\xi^{1,m}\right\rangle d\alpha\\ &\ \ \ \ \ \  +
\int^1_0\int^1_0\left\langle \left(D^2\psi(t,\tilde{x}+\alpha\beta\xi^{2,m})-D^2\psi(t,\tilde{x})\right)\xi^{2,m},\xi^{2,m}\right\rangle\alpha d\beta d\alpha.
\end{align*}
Denote $J^m$ by
\[
J^m:=-\partial_t\psi(t,\tilde{x})\delta+\left\langle D\psi(t,\tilde{x}),\xi^{1,m}\right\rangle+\gamma^{m}\ \text{and}\ \gamma^{m}:=\frac{1}{2}\int^{\delta}_0\widetilde{\sigma}^{D^2\psi(t,\tilde{x})}(\widetilde{X}^{\varepsilon_m,{x}}_s,\overline{X}^{\varepsilon_m,{x}}_s)d\langle B\rangle_s.
\]
In view of the equation \eqref{myw110}, we deduce that
\begin{align*}
\psi(t-\delta,\widetilde{X}^{\varepsilon_m,x}_{\delta})-\psi(t,\tilde{x})=J^{m}+\frac{1}{2}\left\langle D^2\psi(t,\tilde{x})\xi^{2,m},\xi^{2,m}\right\rangle-\gamma^m+\left\langle D\psi(t,\tilde{x}),\xi^{2,m}\right\rangle+\epsilon^{1,m}+\epsilon^{2,m}.
\end{align*}
Note  that $\left\langle D\psi(t,\tilde{x}),\xi^{2,m}\right\rangle$ has no mean uncertainty. Thus, with the help of Lemma \ref{myw619}, we obtain that
\begin{align*}
\left|\mathbb{\hat{E}}[\psi(t-\delta,\widetilde{X}^{\varepsilon_m,x}_{\delta})]-\psi(t,\tilde{x})-\mathbb{\hat{E}}\left[ J^m\right]\right|
=\left|\mathbb{\hat{E}}\left[J^m+\epsilon^{1,m}+\epsilon^{2,m}\right]-\mathbb{\hat{E}}\left[ J^m\right]\right|
\leq \mathbb{\hat{E}}[|\epsilon^{1,m}|+|\epsilon^{2,m}|].
\end{align*}
Note that $\delta<1$. Recalling assumption (H3)  and Lemma \ref{myw6101},   we get that
\begin{align*}
\mathbb{\hat{E}}[|\xi^{1,m}|^2]\leq C(L_1,L_2)(1+|\tilde{x}|^2)|\delta|^{2} \ \text{and}\ \mathbb{\hat{E}}[|\xi^{2,m}|^3]\leq C(L_1,L_2)(1+|\tilde{x}|^3)|\delta|^{\frac{3}{2}}.
\end{align*}
Then, from the definition of $\epsilon^{i,m}$, $i=1,2$,  we conclude that
\[
\mathbb{\hat{E}}[|\epsilon^{1,m}|+|\epsilon^{2,m}|]\leq C(L_1,L_2,\psi)(1+|\tilde{x}|^3)|\delta|^{\frac{3}{2}},
\]
which indicates that
\begin{align}\label{myw909}
\left|\mathbb{\hat{E}}[\psi(t-\delta,\widetilde{X}^{\varepsilon_m,x}_{\delta})]-\psi(t,\tilde{x})-\mathbb{\hat{E}}\left[ J^m\right]\right|
\leq  C(L_1,L_2,\psi)(1+|\tilde{x}|^3)|\delta|^{\frac{3}{2}}.
\end{align}

We claim that
\begin{align}\label{myw1201}
\left|\mathbb{\hat{E}}\left[ J^{m}\right]+H(t,\tilde{x},\psi)\delta\right|\leq  C(L_1,L_2,\eta,\psi)({1+|{x}|^3})\left(\frac{\varepsilon_m}{\delta_m}+\sqrt{\delta_m}+\sqrt{\delta}+\sqrt{\rho_m}\right)\delta,
\end{align}
whose proof will be given in step 4. Here the constants $\delta_m$ and $\rho_m$ will be given in step 3.

From the  inequalities \eqref{myw909} and \eqref{myw1201}, we derive that
\begin{align*}
\left|\mathbb{\hat{E}}[\psi(t-\delta,\widetilde{X}^{\varepsilon_m,x}_{\delta})]-\psi(t,\tilde{x})+H(t,\tilde{x},\psi)\delta\right|\leq C(L_1,L_2,\eta,\psi)(1+|{x}|^3)\left(\frac{\varepsilon_m}{\delta_m}+\sqrt{\delta_m}+\sqrt{\delta}+\sqrt{\rho_m}\right)\delta,
\end{align*}
which implies that
\begin{align*}
\frac{1}{\delta}\left(\mathbb{\hat{E}}[\psi(t-\delta,\widetilde{X}^{\varepsilon_m,x}_{\delta})]-\psi(t,\tilde{x})\right)\leq -H(t,\tilde{x},\psi)+ C(L_1,L_2,\eta,\psi)(1+|{x}|^3)\left(\frac{\varepsilon_m}{\delta_m}+\sqrt{\delta_m}+\sqrt{\delta}+\sqrt{\rho_m}\right).
\end{align*}
Consequently, we put the above inequality into the equation \eqref{myw615}, and obtain that, for each $\delta\in(0,1)$,
\begin{align*}
0\leq \limsup\limits_{m\rightarrow\infty}\frac{1}{\delta}\left(\mathbb{\hat{E}}[\psi(t-\delta,\widetilde{X}^{\varepsilon_m,x}_{\delta})]-\psi(t,\tilde{x})\right)\leq  -H(t,\tilde{x},\psi)+C(L_1,L_2,\eta,\psi)(1+|{x}|^3)\sqrt{\delta},
\end{align*}
where we have used the fact that $\delta_m, \frac{\varepsilon_m}{\delta_m},\rho_m$ converge to $0$ as $m\rightarrow\infty$.
Sending $\delta\rightarrow 0$ yields that $H(t,\tilde{x},\psi)\leq 0$, which is the desired result.

{\bf Step 3} (Khasminskii's discretization). In order to prove the inequality \eqref{myw1201}, we shall introduce Khasminskii's discretization for the fast component $\overline{X}^{\varepsilon,x}$.
For each positive integer $m$, we define   \[\delta_m:=\varepsilon_m\sqrt[4]{\ln{\varepsilon_m^{-1}}}.\] Note that $\delta_m$ converges to $0$ as $m\rightarrow\infty$. Then,  we split the time interval $[0, \delta]$ in length $\delta_m$ for large enough $m$.
Next, consider the following auxiliary stochastic process, for any  $s\in [l\delta_m, (l+1)\delta_m\wedge\delta)$, $l=0,\ldots,N^m$  with $N^m:=[\frac{\delta}{\delta_m}]$,
$$ \overline{X}^{D,\varepsilon_m,x}_s=\overline{X}^{\varepsilon_m,x}_{l\delta_m}+\int_{l\delta_m}^s\frac{\overline{b}(\widetilde{X}^{\varepsilon_m,x}_{l\delta_m},\overline{X}^{D,\varepsilon_m,x}_r)}{\varepsilon_m}dr
+
\int_{l\delta_m}^s\frac{\overline{\sigma} (\widetilde{X}^{\varepsilon_m,x}_{l\delta_m},\overline{X}^{D,\varepsilon_m,x}_r)}{\sqrt{\varepsilon_m}}dB_r, $$
which is well-posed in   light of the assumption (H1).

By Lemma \ref{myw6101} and a standard calculus for $G$-SDEs, we could get that, for any  $s\in [l\delta_m, (l+1)\delta_m\wedge\delta)$,
\begin{align*}
&\hat{\mathbb{E}}\left[\left|\overline{X}^{\varepsilon_m,x}_s-\overline{X}^{D,\varepsilon_m,x}_s\right|^2\right]\\
 & \leq C(L_1)\Big(\frac{\delta_m}{|\varepsilon_m|^2}+\frac{1}{\varepsilon_m}\Big)\int_{l\delta_m}^{s}\left(\hat{\mathbb{E}}\left[\left|\widetilde{ X}^{\varepsilon_m,x}_r-\widetilde{X}^{\varepsilon_m,x}_{l\delta_m}\right|^2\right]+\hat{\mathbb{E}}\left[\left|\overline{X}^{\varepsilon_m,x}_r-\overline{X}^{D,\varepsilon_m,x}_{r}\right|^2\right]\right)dr\\
&\leq C(L_1,L_2)\Big(\frac{\delta_m}{|\varepsilon_m|^2}+\frac{1}{\varepsilon_m}\Big)\left((1+|\tilde{x}|^2)|\delta_m|^2+\int_{l\delta_m}^{s}\hat{\mathbb{E}}\left[\left|\overline{X}^{\varepsilon_m,x}_r-\overline{X}^{D,\varepsilon_m,x}_{r}\right|^2\right]dr\right),
\end{align*}
which together with Gronwall's inequality implies that,
\begin{align}\label{myw610}
\hat{\mathbb{E}}\left[\left|\overline{X}^{\varepsilon_m,x}_s-\overline{X}^{D,\varepsilon_m,x}_s\right|^2\right]
\leq (1+|\tilde{x}|^2)\rho_m
\end{align}
with
\begin{align*}
\rho_m:= C(L_1,L_2)\Big(\frac{\delta_m}{|\varepsilon_m|^2}+\frac{1}{\varepsilon_m}\Big)|\delta_m|^2 \exp{\left(C(L_1,L_2)\left(\frac{\delta_m}{|\varepsilon_m|^2}+\frac{1}{\varepsilon_m}\right)\delta_m\right)}.
\end{align*}
Recalling the choice of  $\delta_m$,  one can easily check that $\rho_m$  converges to $0$ as $m \rightarrow\infty$.

{\bf Step 4} (The proof of the inequality \eqref{myw1201}).
For each $l=0,\ldots,N^m$, set
\begin{align*}
&\xi^{1,m,l}=\int^{(l+1)\delta_m\wedge\delta}_{l\delta_m}\widetilde{b}(\widetilde{X}^{\varepsilon_m,x}_{s},\overline{X}^{\varepsilon_m,{x}}_s)ds, \ \ \xi^{D,1,m,l}=\int^{(l+1)\delta_m\wedge\delta}_{l\delta_m}\widetilde{b}(\widetilde{X}^{\varepsilon_m,x}_{l\delta_m},\overline{X}^{D,\varepsilon_m,{x}}_s)ds, \\
& \gamma^{D,m,l}=\frac{1}{2}\sum\limits_{i,j=1}^d\int^{(l+1)\delta_m\wedge\delta}_{l\delta_m}\widetilde{\sigma}^{D^2\psi(t,\tilde{x})}_{ij}(\widetilde{X}^{\varepsilon_m,x}_{l\delta_m},\overline{X}^{D,\varepsilon_m,{x}}_s)d\langle B^i,B^j\rangle_s,
\end{align*}
and
\[
J^{D,m,l}:=-\partial_t\psi(t,\tilde{x})((l+1)\delta_m\wedge\delta-l\delta_m)+\left\langle D\psi(t,\tilde{x}),\xi^{D,1,m,l}\right\rangle+\gamma^{D,m,l}.
\]

From Lemma \ref{myw6101} and the equation \eqref{myw610}, we get that
\begin{align*}
&\mathbb{\hat{E}}\left[\left|\left\langle D\psi(t,\tilde{x}),\xi^{1,m}\right\rangle-\sum\limits_{l=0}^{N^m}\left\langle D\psi(t,\tilde{x}),\xi^{D,1,m,l}\right\rangle\right|\right]\leq \sum\limits_{l=0}^{N^m}C(\psi) \mathbb{\hat{E}}\left[\left|\xi^{1,m,l}-\xi^{D,1,m,l}\right|\right]
\\&\leq\sum\limits_{l=0}^{N^m}C(L_1,\psi)\int^{(l+1)\delta_m\wedge\delta}_{l\delta_m} \mathbb{\hat{E}}\left[\left|\widetilde{X}^{\varepsilon_m,x}_{s}-\widetilde{X}^{\varepsilon_m,x}_{l\delta_m}\right|+\left|\overline{X}^{\varepsilon_m,{x}}_s-\overline{X}^{D,\varepsilon_m,{x}}_s\right|\right]ds
 \\
&\leq \sum\limits_{l=0}^{N^m}C(L_1,L_2,\psi)(1+|\tilde{x}|)\left(\sqrt{\delta_m}+\sqrt{\rho_m}\right)\delta_m\leq C(L_1,L_2,\psi)(1+|\tilde{x}|)\left(\sqrt{\delta_m}+\sqrt{\rho_m}\right)\delta,
\end{align*}
where we have used the fact that $N^m\delta_m\leq\delta\leq (N^m+1)\delta_m$ in the last inequality. By a similar way, we
could obtain that
\[
\mathbb{\hat{E}}\left[\left| \gamma^m-\sum\limits_{l=0}^{N^m}\gamma^{D,m,l}\right|\right]\leq  C(L_1,L_2,\psi)(1+|\tilde{x}|^2)\left(\sqrt{\delta_m}+\sqrt{\rho_m}\right)\delta.
\]
Thus, it follows from the definition of $J^m$ that
\begin{align}\label{myw801}
\begin{split}
\mathbb{\hat{E}}\left[\left|J^{m}-\sum\limits_{l=0}^{N^m}J^{D,m,l}\right|\right]=&
\mathbb{\hat{E}}\left[\left|\left\langle D\psi(t,\tilde{x}),\xi^{1,m}\right\rangle+\gamma^m-\sum\limits_{l=0}^{N^m}\left\langle D\psi(t,\tilde{x}),\xi^{D,1,m,l}\right\rangle-\sum\limits_{l=0}^{N^m}\gamma^{D,m,l}\right|\right]\\
\leq& C(L_1,L_2,\psi)(1+|\tilde{x}|^2)\left(\sqrt{\delta_m}+\sqrt{\rho_m}\right)\delta.
\end{split}
\end{align}

Now, with the help of Assertion (iv) of Lemma \ref{myw502} and Lemma \ref{myw6101}, we have that,
\begin{align*}
&\mathbb{\hat{E}}\left[\left|\widetilde{G}(\widetilde{X}^{\varepsilon_m,x}_{l\delta_m},D\psi(t,\tilde{x}),D^2\psi(t,\tilde{x}))-\widetilde{G}(\tilde{x},D\psi(t,\tilde{x}),D^2\psi(t,\tilde{x}))\right|\right]\\
 & \leq C(L_1,L_2,\eta)\left(\left|D\psi(t,\tilde{x})\right|+\left|D^2\psi(t,\tilde{x})\right|\right)\mathbb{\hat{E}}\left[\left(1+\left|\widetilde{X}^{\varepsilon_m,x}_{l\delta_m}\right|^2+|\tilde{x}|^2\right)\left|\widetilde{X}^{\varepsilon_m,x}_{l\delta_m}-\tilde{x}\right|\right]\\
&\leq C(L_1,L_2,\eta,\psi)\left(\mathbb{\hat{E}}\left[\left(1+\left|\widetilde{X}^{\varepsilon_m,x}_{l\delta_m}\right|^2+|\tilde{x}|^2\right)^2\right]\mathbb{\hat{E}}\left[\left|\widetilde{X}^{\varepsilon_m,x}_{l\delta_m}-\tilde{x}\right|^2\right]\right)^{\frac{1}{2}}\leq  C(L_1,L_2,\eta,\psi)(1+|\tilde{x}|^3)\sqrt{\delta}.
\end{align*}
On the other hand, applying Lemma \ref{myw603} and Lemma \ref{myw202}  yields that, \begin{align}\label{myq206136}\begin{split}
&\mathbb{\hat{E}}\left[\left|\mathbb{\hat{E}}_{l\delta_m}\left[ \left\langle D\psi(t,\tilde{x}),\xi^{D,1,m,l}\right\rangle+\gamma^{D,m,l} \right]-\widetilde{G}(\widetilde{X}^{\varepsilon_m,x}_{l\delta_m},D\psi(t,\tilde{x}),D^2\psi(t,\tilde{x}))((l+1)\delta_m\wedge\delta-l\delta_m)\right|\right]
\\
&\leq C(L_1,L_2,\eta,\psi)\mathbb{\hat{E}}\left[\left(1+|\widetilde{X}^{\varepsilon_m,x}_{l\delta_m}|^2+|\overline{X}^{\varepsilon_m,x}_{l\delta_m}|^2\right)\right]\varepsilon_m\leq C(L_1,L_2,\eta,\psi)(1+|x|^2)\varepsilon_m.
\end{split}
\end{align}
Then, by the definition of $H$ (see inequality \eqref{myw612}) and the above two inequalities, we derive  that, for each $l=0,1,\ldots,N^m$,
\begin{align*}
&\mathbb{\hat{E}}\left[\left|\mathbb{\hat{E}}_{l\delta_m}\left[ J^{D,m,l} \right]+H(t,\tilde{x},\psi)((l+1)\delta_m\wedge\delta-l\delta_m)\right|\right]\\
&=\mathbb{\hat{E}}\left[\left|\mathbb{\hat{E}}_{l\delta_m}\left[ \left\langle D\psi(t,\tilde{x}),\xi^{D,1,m,l}\right\rangle+\gamma^{D,m,l} \right]-\widetilde{G}(\tilde{x},D\psi(t,\tilde{x}),D^2\psi(t,\tilde{x}))((l+1)\delta_m\wedge\delta-l\delta_m)\right|\right]\\
&\leq
\mathbb{\hat{E}}\left[\left|\mathbb{\hat{E}}_{l\delta_m}\left[ \left\langle D\psi(t,\tilde{x}),\xi^{D,1,m,l}\right\rangle+\gamma^{D,m,l} \right]-\widetilde{G}(\widetilde{X}^{\varepsilon_m,x}_{l\delta_m},D\psi(t,\tilde{x}),D^2\psi(t,\tilde{x}))((l+1)\delta_m\wedge\delta-l\delta_m)\right|\right]\\
&\ \ \ \ \  +\mathbb{\hat{E}}\left[\left|\widetilde{G}(\widetilde{X}^{\varepsilon_m,x}_{l\delta_m},D\psi(t,\tilde{x}),D^2\psi(t,\tilde{x}))-\widetilde{G}(\tilde{x},D\psi(t,\tilde{x}),D^2\psi(t,\tilde{x}))\right|\right]((l+1)\delta_m\wedge\delta-l\delta_m)\\
&\leq C(L_1,L_2,\eta,\psi)(1+|x|^3)\left(\varepsilon_m+\sqrt{\delta}((l+1)\delta_m\wedge\delta-l\delta_m)\right).
\end{align*}
It follows that
\begin{align*}
&\left|\mathbb{\hat{E}}\left[ \sum\limits_{l=0}^{N^m}J^{D,m,l}\right]+H(t,\tilde{x},\psi)\delta\right|=
\left|\mathbb{\hat{E}}\left[ \sum\limits_{l=0}^{N^m-1}J^{D,m,l}+\mathbb{\hat{E}}_{N^m\delta}\left[ J^{D,m,N^m} \right]\right]+H(t,\tilde{x},\psi)\delta\right|\\
&\leq
\left|\mathbb{\hat{E}}\left[ \sum\limits_{l=0}^{N^m-1}J^{D,m,l}\right]+H(t,\tilde{x},\psi)N^m\delta_m\right|+\mathbb{\hat{E}}\left[\left|\mathbb{\hat{E}}_{N^m\delta}\left[ J^{D,m,N^m} \right]+H(t,\tilde{x},\psi)(\delta-N^m\delta_m)\right|\right]
\\ &\leq  \cdots \cdots \cdots
\\ &\leq \sum\limits_{l=0}^{N^m}\mathbb{\hat{E}}\left[\left|\mathbb{\hat{E}}_{l\delta_m}\left[ J^{D,m,l} \right]+H(t,\tilde{x},\psi)((l+1)\delta_m\wedge\delta-l\delta_m)\right|\right]\leq  C(L_1,L_2,\eta,\psi)(1+|x|^3)\left(\frac{\varepsilon_m}{\delta_m}+\sqrt{\delta}\right)\delta.
\end{align*}
Consequently, in view of the inequality \eqref{myw801}, we deduce that
\[
\left|\mathbb{\hat{E}}\left[ J^{m}\right]+H(t,\tilde{x},\psi)\delta\right|\leq C(L_1,L_2,\eta,\psi)({1+|{x}|^3})\left(\frac{\varepsilon_m}{\delta_m}+\sqrt{\delta_m}+\sqrt{\delta}+\sqrt{\rho_m}\right)\delta.
\]
The proof is complete.
\end{proof}
\begin{remark}\label{myw9019}{\upshape
The assumption (H3) is used to establish a uniform $L^3_G$-estimate of the slow diffusion process, which is crucial for inequality \eqref{myw909} in our setting. Indeed, one could strengthen the assumption (H2) to remove (H3); see Lemma 3.2 in \cite{HW}.
}
\end{remark}
\begin{lemma}\label{myw202}
Suppose that $(p,A)$ is in $\mathbb{R}^n\times\mathbb{S}(n)$. Then, for each $s\in [l\delta_m,(l+1)\delta_m\wedge\delta]$, $l=0,\ldots,N^m$, it holds that
\begin{align*}
&\left|\mathbb{\hat{E}}_{l\delta_m}\left[{\int^{s}_{l\delta_m}\langle p,  \widetilde{b}(\widetilde{X}^{\varepsilon_m,x}_{l\delta_m},\overline{X}^{D,\varepsilon_m,{x}}_r)\rangle dr}+\frac{1}{2}\int^{s}_{l\delta_m}\widetilde{\sigma}^A(\widetilde{X}^{\varepsilon_m,x}_{l\delta_m},\overline{X}^{D,\varepsilon_m,{x}}_r)d\langle B\rangle_r\right]-\widetilde{G}(\widetilde{X}^{\varepsilon,x}_{l\delta_m},p,A) (s-l\delta_m)\right|\\
&\leq C(L_1,L_2,\eta) (|p|+|A|)\left(1+|\widetilde{X}^{\varepsilon_m,x}_{l\delta_m}|^2+|\overline{X}^{\varepsilon_m,x}_{l\delta_m}|^2\right)\varepsilon_m.
\end{align*}
\end{lemma}
\begin{proof}
Consider the following $G$-SDE:  for each $x=(\tilde{x},\bar{x})\in\mathbb{R}^{2n}$,
\begin{align*}
 \overline{X}^{\prime,\varepsilon,(\tilde{x},\bar{x})}_t=\bar{x}+\int^t_0\frac{\overline{b}(\tilde{x},\overline{X}^{\prime,\varepsilon,(\tilde{x},\bar{x})}_r)}{\varepsilon}dr+\int^t_0\frac{\overline{\sigma}(\tilde{x},\overline{X}^{\prime,\varepsilon,(\tilde{x},\bar{x})}_r)}{\sqrt{\varepsilon}}dB_r.
\end{align*}
Recalling equation \eqref{myw501}, we get that
\begin{align*}
\overline{X}^{(\tilde{x},\bar{x})}_{\frac{t}{\varepsilon_m}}
=&\bar{x}+\frac{1}{\varepsilon_m}\int^{t}_0\overline{b}(\tilde{x},\overline{X}^{(\tilde{x},\bar{x})}_{\frac{r}{\varepsilon_m}})dr+{\frac{1}{\sqrt{\varepsilon_m}}}\int^{t}_0\overline{\sigma}(\tilde{x},\overline{X}^{(\tilde{x},\bar{x})}_{\frac{r}{\varepsilon_m}})dB^{\varepsilon_m}_r,
\end{align*}
where $(B^{\varepsilon_m}_r=\sqrt{\varepsilon_m} B_{\frac{r}{\varepsilon_m}})_{r\geq 0}$ also is a $G$-Brownian motion.
Then, by a standard approximation method, we derive that $(\overline{X}^{\prime,\varepsilon_m,(\tilde{x},\bar{x})}_{{t}},B_t)$ has the same distribution
as $(\overline{X}^{(\tilde{x},\bar{x})}_{\frac{t}{\varepsilon_m}},B_t^{\varepsilon_m})$. It follows that
\begin{align*}
&\mathbb{\hat{E}}\left[{\int^{t}_0\langle p,  \widetilde{b}(\tilde{x},\overline{X}^{\prime,\varepsilon_m,(\tilde{x},\bar{x})}_r)\rangle dr}+\frac{1}{2}\sum_{i,j=1}^d\int^{t}_0\widetilde{\sigma}^A_{ij}(\tilde{x},\overline{X}^{\prime,\varepsilon_m,(\tilde{x},\bar{x})}_r)d\langle B^i, B^j\rangle_r\right]\\
&=\mathbb{\hat{E}}\left[{\int^{t}_0\langle p,  \widetilde{b}(\tilde{x},\overline{X}^{(\tilde{x},\bar{x})}_{\frac{r}{\varepsilon_m}})\rangle dr}+\frac{1}{2}\sum_{i,j=1}^d\int^{t}_0\widetilde{\sigma}^A_{ij}(\overline{X}^{(\tilde{x},\bar{x})}_{\frac{r}{\varepsilon_m}})d\langle B^{\varepsilon_m,i}, B^{\varepsilon_m,j}\rangle_r\right]\\
&=\varepsilon_m\mathbb{\hat{E}}\left[{\int^{\frac{t}{\varepsilon_m}}_0\langle p,  \widetilde{b}(\tilde{x},\overline{X}^{{x}}_{r})\rangle dr}+\frac{1}{2}\sum_{i,j=1}^d\int^{\frac{t}{\varepsilon_m}}_0\widetilde{\sigma}^A_{ij}(\tilde{x},\overline{X}^{{x}}_{r})d\langle B^{i}, B^{j}\rangle_r\right],
\end{align*}
which together with inequality \eqref{myw125} yields  that, for each $t>0$,
\begin{align}\begin{split}\label{myw803}
&\left|\mathbb{\hat{E}}\left[{\int^{t}_0\langle p,  \widetilde{b}(\tilde{x},\overline{X}^{\prime,\varepsilon_m,(\tilde{x},\bar{x})}_r)\rangle dr}+\frac{1}{2}\sum_{i,j=1}^d\int^{t}_0\widetilde{\sigma}^A_{ij}(\tilde{x},\overline{X}^{\prime,\varepsilon_m,(\tilde{x},\bar{x})}_r)d\langle B^i, B^j\rangle_r\right]-\widetilde{G}(\tilde{x},p,A) t\right|
\\
& \leq \left|\mathbb{\hat{E}}\left[{\int^{\frac{t}{\varepsilon_m}}_0\langle p,  \widetilde{b}(\tilde{x},\overline{X}^{{x}}_{r})\rangle dr}+\frac{1}{2}\sum_{i,j=1}^d\int^{\frac{t}{\varepsilon_m}}_0\widetilde{\sigma}^A_{ij}(\tilde{x},\overline{X}^{{x}}_{r})d\langle B^{i}, B^{j}\rangle_r\right]-\widetilde{G}(\tilde{x},p,A) \frac{t}{\varepsilon_m}\right|\varepsilon_m
\\
& \leq C(L_1,L_2,\eta)(|p|+|A|)(1+|\tilde{x}|^2+|\bar{x}|^2) \varepsilon_m.
\end{split}
\end{align}

On the other hand, recalling the definition $\overline{X}^{D,\varepsilon_m,x}$ and using the Markov property (see Assertion (7) of Theorem 5.1 in \cite{HJPS1}), we conclude that
\begin{align*}
\begin{split}
&\mathbb{\hat{E}}_{l\delta_m}\left[{\int^{s}_{l\delta_m}\langle p,  \widetilde{b}(\widetilde{X}^{\varepsilon_m,x}_{l\delta_m},\overline{X}^{D,\varepsilon_m,{x}}_r)\rangle dr}+\frac{1}{2}\sum_{i,j=1}^d\int^{s}_{l\delta_m}\widetilde{\sigma}^A_{ij}(\widetilde{X}^{\varepsilon_m,x}_{l\delta_m},\overline{X}^{D,\varepsilon_m,{x}}_r)d\langle B^i, B^j\rangle_r\right]\\
&=\mathbb{\hat{E}}\left[\int^{s-l\delta_m}_{0}\left\langle p,\widetilde{b}(\tilde{x}^{\prime},\overline{X}^{\prime,\varepsilon_m,(\tilde{x}^{\prime},\bar{x}^{\prime})}_{r})\right\rangle dr+\frac{1}{2}\int^{s-l\delta_m}_{0}\widetilde{\sigma}^{A}(\tilde{x}^{\prime},\overline{X}^{\prime,\varepsilon_m,(\tilde{x}^{\prime},\bar{x}^{\prime})}_{r})d\langle B\rangle_r\right]_{(\tilde{x}^{\prime},\bar{x}^{\prime})=(\widetilde{X}^{\varepsilon_m,x}_{l\delta_m},\overline{X}^{\varepsilon_m,x}_{l\delta_m})},
\end{split}
\end{align*}
which together with the inequality \eqref{myw803} indicates the desired result. This ends the proof.
\end{proof}

Finally, we are ready to state the proofs of Theorem \ref{myw521} and Theorem \ref{myw1906}.

\begin{proof}[The proof of Theorem \ref{myw521}]
 Without loss of generality, assume that $t\in[0,T]$.
 Let $(t,\tilde{x},\bar{x})\in[0,T]\times\mathbb{R}^{2n}$ be fixed. Denote by $\widetilde{u}^{\varphi}$ the solution to PDE \eqref{PDE511} with the initial condition
$\varphi.$  Similarly, we can define $u^{\varphi,\varepsilon}$. The proof is  divided into the following two steps.

{\bf  Step 1} ($\varphi\in C_{b.lip}(\mathbb{R}^n)$).
Suppose that the sequence $(\varepsilon_l)_{l\geq 1}$ converges to $0$.
Then, from Lemma \ref{myw608}, we can find a subsequence
$(\varepsilon_{l_m})_{m\geq 1}$ such that ${u}^{\varphi,\varepsilon_{l_m}}$ converges to some  function $\widetilde{u}^{\varphi,*}\in C_b([0,T]\times\mathbb{R}^n)$ on $[0,T]\times\mathbb{R}^{2n}$.
Applying Lemma \ref{myw611} yields that $\widetilde{u}^{\varphi,*}$ is a viscosity solution to the averaged PDE \eqref{PDE511}. It follows from  Lemma \ref{myw507} in appendix B that $\widetilde{u}^{\varphi,*}\equiv\widetilde{u}^{\varphi}$. Thus, we derive that
\[
\lim\limits_{m\rightarrow\infty}{u}^{\varphi,\varepsilon_{l_m}}(t,\tilde{x},\bar{x})=\widetilde{u}^{\varphi}(t,\tilde{x}),
\]
which implies the desired result.

{\bf Step 2 }($\varphi\in C(\mathbb{R}^n)$ of polynomial growth).
For each positive integer $N$, we can find a function $\varphi_N\in C_{b.lip}(\mathbb{R}^n)$ so that
\[
|\varphi_N(\tilde{x}^{\prime})-\varphi(\tilde{x}^{\prime})|\leq C(\varphi)\frac{1+|\tilde{x}^{\prime}|^{ C(\varphi)}}{N}, \ \forall \tilde{x}^{\prime}\in\mathbb{R}^n.
\]
With the help of Lemma \ref{myw6101}, we have that, for any $(s,x^{\prime})\in[0,T]\times\mathbb{R}^{2n}$,
\begin{align}\label{myw905}
\mathbb{\hat{E}}\left[\left|\varphi_N(\widetilde{X}^{\varepsilon,{x}^{\prime}}_s)-\varphi(\widetilde{X}^{\varepsilon,{x}^{\prime}}_s)\right|\right]\leq
 C(\varphi)\frac{1+\mathbb{\hat{E}}\left[\left|\widetilde{X}^{\varepsilon,{x}^{\prime}}_s\right|^{ C(\varphi)}\right]}{N}\leq C(L_1,L_2,T,\varphi)\frac{1+|\tilde{x}^{\prime}|^{ C(\varphi)}}{N}.
\end{align}
Then, from inequality \eqref{myw905}, we get that
\begin{align*}
&\liminf\limits_{\varepsilon\rightarrow 0}\mathbb{\hat{E}}\left[\varphi(\widetilde{X}^{\varepsilon,{x}^{\prime}}_s)\right]\geq-C(L_1,L_2,T,\varphi)\frac{1+|\tilde{x}^{\prime}|^{ C(\varphi)}}{N}+\lim\limits_{\varepsilon\rightarrow 0}\mathbb{\hat{E}}\left[\varphi_N(\widetilde{X}^{\varepsilon,{x}^{\prime}}_s)\right],\\
&\limsup\limits_{\varepsilon\rightarrow 0}\mathbb{\hat{E}}\left[\varphi(\widetilde{X}^{\varepsilon,{x}^{\prime}}_s)\right]\leq C(L_1,L_2,T,\varphi)\frac{1+|\tilde{x}^{\prime}|^{ C(\varphi)}}{N}+\lim\limits_{\varepsilon\rightarrow 0}\mathbb{\hat{E}}\left[\varphi_N(\widetilde{X}^{\varepsilon,{x}^{\prime}}_s)\right],
\end{align*}
which implies that
\[
\limsup\limits_{\varepsilon\rightarrow 0}\mathbb{\hat{E}}\left[\varphi(\widetilde{X}^{\varepsilon,{x}^{\prime}}_s)\right]\leq\liminf\limits_{\varepsilon\rightarrow 0}\mathbb{\hat{E}}\left[\varphi(\widetilde{X}^{\varepsilon,{x}^{\prime}}_s)\right]+C(L_1,L_2,T,\varphi)\frac{1+|\tilde{x}^{\prime}|^{ C(\varphi)}}{N}.
\]
Sending $N\rightarrow\infty$, we deduce that $\lim\limits_{\varepsilon\rightarrow 0}\mathbb{\hat{E}}\left[\varphi(\widetilde{X}^{\varepsilon,{x}^{\prime}}_s)\right]$ exists. Moreover, recalling equation \eqref{myw905}, we obtain that
\[
\lim\limits_{\varepsilon\rightarrow 0}\mathbb{\hat{E}}\left[\varphi(\widetilde{X}^{\varepsilon,{x}^{\prime}}_s)\right]=\lim\limits_{N\rightarrow \infty}
\lim\limits_{\varepsilon\rightarrow 0}\mathbb{\hat{E}}\left[\varphi_N(\widetilde{X}^{\varepsilon,{x}^{\prime}}_s)\right]=\lim\limits_{N\rightarrow \infty}\widetilde{u}^{\varphi_N}(s,\tilde{x}^{\prime})=:\widetilde{u}^{\varphi,*}(s,\tilde{x}^{\prime}).
\]

On the other hand, applying equation \eqref{myw905} again, we get that
\[
|\widetilde{u}^{\varphi_N}(s,\tilde{x}^{\prime})-\widetilde{u}^{\varphi,*}(s,\tilde{x}^{\prime})|=\lim\limits_{\varepsilon\rightarrow 0}\left|\mathbb{\hat{E}}\left[\varphi_N(\widetilde{X}^{\varepsilon,{x}^{\prime}}_s)\right]-\mathbb{\hat{E}}\left[\varphi(\widetilde{X}^{\varepsilon,{x}^{\prime}}_s)\right]\right|\leq C(\varphi)\frac{1+|\tilde{x}^{\prime}|^{ C(\varphi)}}{N},
\]
which indicates that $\widetilde{u}^{\varphi_N}$ converges uniformly to $\widetilde{u}^{\varphi,*}$ on each compact subset of $[0,T]\times\mathbb{R}^{n}$. In the spirit of Proposition 4.3 in \cite{CMI} and Lemma \ref{myw507} in appendix B, we conclude that $\widetilde{u}^{\varphi,*}$ is the unique viscosity solution
to the averaged PDE \eqref{PDE511}. The proof is complete.
\end{proof}

\begin{proof}[The proof of Theorem \ref{myw1906}]
It suffices to prove the case that $k=2$, since other cases can be proved by iterative method. Without loss of generality, assume that
$t_1,t_2\in[0,T]$.

From Theorem \ref{myw521} and Lemma \ref{myw6101}, it is easy to check that
$\varphi ^{1}(\tilde{x}^1) =\lim\limits_{\varepsilon\rightarrow0}\mathbb{\hat{E}}[\varphi(\tilde{x}^1,\widetilde{X}^{\varepsilon,(\tilde{x}^{1},0)}_{t_2-t_{1}})]
$ is well-defined and  of polynomial growth.
We claim that
\begin{align*}
\text{$(\tilde{x}^{1},\bar{x}^1)\rightarrow\mathbb{\hat{E}}\left[\varphi(\tilde{x}^1,\widetilde{X}^{\varepsilon,(\tilde{x}^{1},\bar{x}^1)}_{t_2-t_{1}})\right]$ uniformly converges to $\varphi ^{1}(\tilde{x}^1)$ on  each compact subset of $\mathbb{R}^{2n}$},
\end{align*}
whose proof will be given later.

According to the Markov property (see \cite{HJPS1}) and Lemma \ref{myw6101}, we conclude that,
\begin{align*}
&\left|\mathbb{\hat{E}}\left[\varphi(\widetilde{X}^{\varepsilon,x}_{t_1},\widetilde{X}^{\varepsilon,x}_{t_2})\right]-\mathbb{\hat{E}}\left[\varphi^{1}(\widetilde{X}^{\varepsilon,x}_{t_1})\right]\right|\leq \mathbb{\hat{E}}\left[\left|\mathbb{\hat{E}}\left[\varphi(\tilde{x}^1,\widetilde{X}^{\varepsilon,(\tilde{x}^1,\bar{x}^1)}_{t_2-t_1})\right]_{(\tilde{x}^1,\bar{x}^1)=(\widetilde{X}^{\varepsilon,x}_{t_1},\overline{X}^{\varepsilon,x}_{t_1})}
-\varphi^{1}(\widetilde{X}^{\varepsilon,x}_{t_1})\right|\right]
\\&
\leq \sup\limits_{|\tilde{x}^1|,|\bar{x}^1|\leq N}\left|\mathbb{\hat{E}}\left[\varphi(\tilde{x}^1,\widetilde{X}^{\varepsilon,(\tilde{x}^1,\bar{x}^1)}_{t_2-t_1})\right]
-\varphi^{1}(\tilde{x}^1)\right|+\frac{C(L_1,L_2,T,\varphi)}{N}\mathbb{\hat{E}}\left[\left(1+\left|\widetilde{X}^{\varepsilon,x}_{t_1}\right|^{C(\varphi)}+\left|\overline{X}^{\varepsilon,x}_{t_1}\right|\right)\right]\\&
\leq \sup\limits_{|\tilde{x}^1|,|\bar{x}^1|\leq N}\left|\mathbb{\hat{E}}\left[\varphi(\tilde{x}^1,\widetilde{X}^{\varepsilon,(\tilde{x}^1,\bar{x}^1)}_{t_2-t_1})\right]
-\varphi^{1}(\tilde{x}^1)\right|+\frac{C(L_1,L_2,T,\varphi)}{N}(1+|x|^{C(\varphi)}), \ \ \forall N\geq 1,
\end{align*}
where we have used Lemma \ref{myw603} in the last inequality. It follows that
\[
\limsup\limits_{\varepsilon\rightarrow0}\left|\mathbb{\hat{E}}\left[\varphi(\widetilde{X}^{\varepsilon,x}_{t_1},\widetilde{X}^{\varepsilon,x}_{t_2})\right]-\mathbb{\hat{E}}\left[\varphi^{1}(\widetilde{X}^{\varepsilon,x}_{t_1})\right]\right|\leq \frac{C(L_1,L_2,T,\varphi)}{N}(1+|x|^{C(\varphi)}).
\]
Sending $N\rightarrow 0$, we obtain that
\[
\lim\limits_{\varepsilon\rightarrow0}\mathbb{\hat{E}}\left[\varphi(\widetilde{X}^{\varepsilon,x}_{t_1},\widetilde{X}^{\varepsilon,x}_{t_2})\right]=\lim\limits_{\varepsilon\rightarrow0}\mathbb{\hat{E}}\left[\varphi^{1}(\widetilde{X}^{\varepsilon,x}_{t_1})\right],
\]
which is the desired result.

Now, we shall prove the above claim.
For each positive integer $N$, we can find a function $\varphi_N\in C_{b.lip}(\mathbb{R}^{2n})$ so that
\[
|\varphi_N(\tilde{x}^{1},\tilde{x}^2)-\varphi(\tilde{x}^{1},\tilde{x}^2)|\leq C(\varphi)\frac{1+|\tilde{x}^{1}|^{ C(\varphi)}+|\tilde{x}^{2}|^{ C(\varphi)}}{N}, \ \forall \tilde{x}^{1},\tilde{x}^2\in\mathbb{R}^n.
\]
 Thus, according to Lemma \ref{myw6101}, we have that
\begin{align*}
&\left|\mathbb{\hat{E}}\left[\varphi(\tilde{x}^1,\widetilde{X}^{\varepsilon,(\tilde{x}^{1},\bar{x}^1}_{t_2-t_{1}})\right]-\varphi^{1}(\tilde{x}^1)\right|
\\
&\leq \left|\mathbb{\hat{E}}\left[\varphi(\tilde{x}^1,\widetilde{X}^{\varepsilon,(\tilde{x}^{1},\bar{x}^1)}_{t_2-t_{1}})\right]-\mathbb{\hat{E}}\left[\varphi_N(\tilde{x}^1,\widetilde{X}^{\varepsilon,(\tilde{x}^{1},\bar{x}^1)}_{t_2-t_{1}})\right]\right|+\left|\mathbb{\hat{E}}[\varphi_N(\tilde{x}^1,\widetilde{X}^{\varepsilon,(\tilde{x}^{1},\bar{x}^1)}_{t_2-t_{1}})]-\varphi^1_N(\tilde{x}^1)\right|+\left|\varphi^1_N(\tilde{x}^1)-\varphi^1(\tilde{x}^1)\right|
\\
&\leq \left|\mathbb{\hat{E}}\left[\varphi_N(\tilde{x}^1,\widetilde{X}^{\varepsilon,(\tilde{x}^{1},\bar{x}^1)}_{t_2-t_{1}})\right]-\varphi^1_N(\tilde{x}^1)\right|+
C(L_1,L_2,T,\varphi)\frac{1+|\tilde{x}^1|^{ C(\varphi)}}{N}.
\end{align*}
On the other hand,
by a similar analysis as in Lemma \ref{myw608}, it is easy to check that $\mathbb{\hat{E}}[\varphi_N(\tilde{x}^1,\widetilde{X}^{\varepsilon,(\tilde{x}^{1},\bar{x}^1)}_{t_2-t_{1}})]$ uniformly converges to $\varphi ^{1}_N(\tilde{x}^1)$ on  each compact subset of $\mathbb{R}^{2n}$. Consequently, from the above inequality, we get that the desired claim holds. The proof is complete.
\end{proof}

\section{Extension}
In the previous sections, we develop a useful approach to the study of averaging of $G$-SDEs with two time scales by using nonlinear stochastic analysis and viscosity solution theory.
We would like to mention that our main ideas carry over to much
more general frameworks. In this section, we will extend the previous results
to a more general case.

Consider the following fast-slow scale diffusion process (see \cite{PS}):
\begin{align*}
\begin{cases}
&{\displaystyle \widetilde{S}^{\varepsilon,x}_t=\tilde{x}+\int^t_0\widetilde{b}(\widetilde{S}^{\varepsilon,x}_s,\overline{S}^{\varepsilon,x}_s)ds+\sum\limits_{i,j=1}^d\int^t_0\widetilde{h}_{ij}(\widetilde{S}^{\varepsilon,x}_s,\overline{S}^{\varepsilon,x}_s)d\langle B^i,B^j\rangle_s+\int^t_0\widetilde{\sigma}(\widetilde{S}^{\varepsilon,x}_s,\overline{S}^{\varepsilon,x}_s)dB_s,}\\
&{ \displaystyle \overline{S}^{\varepsilon,x}_t=\bar{x}+\int^t_0\left(\frac{\overline{b}(\widetilde{S}^{\varepsilon,x}_s,\overline{S}^{\varepsilon,x}_s)}{\varepsilon}+\frac{\overline{b}^1(\widetilde{S}^{\varepsilon,x}_s,\overline{S}^{\varepsilon,x}_s)}{\sqrt{\varepsilon}}+\overline{b}^2(\widetilde{S}^{\varepsilon,x}_s,\overline{S}^{\varepsilon,x}_s)\right)ds}\\
&{\displaystyle \ \ \ \ \ \ \ \ \ \ \ \   +\sum\limits_{i,j=1}^d\int^t_0\left(\frac{\overline{h}_{ij}(\widetilde{S}^{\varepsilon,x}_s,\overline{S}^{\varepsilon,x}_s)}{\varepsilon}+\frac{\overline{h}^1_{ij}(\widetilde{S}^{\varepsilon,x}_s,\overline{S}^{\varepsilon,x}_s)}{\varepsilon}+\overline{h}^2_{ij}(\widetilde{S}^{\varepsilon,x}_s,\overline{S}^{\varepsilon,x}_s)\right)d\langle B^i,B^j\rangle_s}\\
&{\displaystyle
\ \ \ \ \ \ \ \ \ \ \ \ +\int^t_0\left(\frac{\overline{\sigma}(\widetilde{S}^{\varepsilon,x}_s,\overline{S}^{\varepsilon,x}_s)}{\sqrt{\varepsilon}}+\overline{\sigma}^1(\widetilde{S}^{\varepsilon,x}_s,\overline{S}^{\varepsilon,x}_s)\right)dB_s,}
\end{cases}
\end{align*}
for each $x=(\tilde{x},\bar{x})\in\mathbb{R}^{2n}$,
where $\overline{b}^1$, $\overline{b}^2$, $\overline{h}^1_{ij}=\overline{h}^1_{ji}$, $\overline{h}^2_{ij}=\overline{h}^2_{ji}:\mathbb{R}^{2n}\rightarrow \mathbb{R}^{n}$, $\overline{\sigma}^1:\mathbb{R}^{2n}\rightarrow \mathbb{R}^{n\times d}$ are deterministic non-periodic functions satisfying the following.
\begin{description}
\item[(H4)]There exists a constant $L_1>0$ such that, for any $x,x^{\prime}\in\mathbb{R}^{2n}$,
\begin{align*}
|{\ell}(x)-{\ell}(x^{\prime})|\leq
L_1|x-x^{\prime}|\ \text{and}\ |\ell(0)|\leq L_1,\ \text{for $\ell=\overline{b}^1, \overline{b}^2, \overline{h}^1_{ij}, \overline{h}^2_{ij}$, and $\overline{\sigma}^1$.}
\end{align*}
\end{description}

Note that the new coefficients added vary slowly compared with the original ones, and therefore they do not affect the structure of the averaged PDEs as in \cite{PS}; see Theorem \ref{myw2020615}.

\begin{lemma}\label{myw60313}
Assume   \emph{(H1)}, \emph{(H2)} and \emph{(H4)} hold. Then, there exists a constant $C(L_1,\eta,T)$, such that for each $x\in\mathbb{R}^{2n}$,
\begin{description}
	\item[(i)]$\mathbb{\hat{E}}\left[\sup\limits_{0\leq t\leq T}\left|\widetilde{S}^{\varepsilon,x}_t\right|^2\right]+\sup\limits_{0\leq t\leq T}\mathbb{\hat{E}}\left[\left|\overline{S}^{\varepsilon,x}_t\right|^2\right]\leq  C(L_1,\eta,T)\left(1+|{x}|^2\right)$,
	\item[(ii)] $\mathbb{\hat{E}}\left[\sup\limits_{0\leq t\leq T}\left|\widetilde{S}^{\varepsilon,x}_t-\widetilde{X}^{\varepsilon,x}_t\right|^2\right]\leq C(L_1,\eta,T)(1+|x|^2)\varepsilon$.
\end{description}
\end{lemma}
\begin{proof}
 We shall only  give the sketch of the proof for readers' convenience. Without loss of generality, assume that $\widetilde{h}_{ij}=\overline{h}_{ij}=\overline{h}^1_{ij}=\overline{h}^2_{ij}=0$, for $i,j=1,\ldots,d.$  The proof is divided into the following two steps.

{\bf Step 1} (Assertion {(i)}). Applying $G$-It\^{o}'s formula, we deduce that
\begin{align}\label{myw60578}\begin{split}
&e^{\frac{\eta}{\varepsilon}t}\left|\overline{S}^{\varepsilon,x}_t\right|^2-|\bar{x}|^2-M^1_t-\frac{\eta}{\varepsilon}\int^t_0e^{\frac{\eta}{\varepsilon}s}\left|\overline{S}^{\varepsilon,x}_s\right|^2ds\\
&\leq
\frac{2}{\varepsilon}\int^t_0e^{\frac{\eta}{\varepsilon}s}\left[\left\langle\overline{S}^{\varepsilon,x}_s,\overline{b}(\widetilde{S}^{\varepsilon,x}_s,\overline{S}^{\varepsilon,x}_s)\right\rangle+G\left(\left(\overline{\sigma}(\widetilde{S}^{\varepsilon,x}_s,\overline{S}^{\varepsilon,x}_s)\right)^{\top}\overline{\sigma}(\widetilde{S}^{\varepsilon,x}_s,\overline{S}^{\varepsilon,x}_s)\right)\right]ds\\
&\ \ \ + \frac{2}{\sqrt{\varepsilon}}\int^t_0e^{\frac{\eta}{\varepsilon}s}\left[\left\langle\overline{S}^{\varepsilon,x}_s,\overline{b}^1(\widetilde{S}^{\varepsilon,x}_s,\overline{S}^{\varepsilon,x}_s)\right\rangle +G\left(\left(\left(\overline{\sigma}^1\right)^{\top}\overline{\sigma}+\left(\overline{\sigma}\right)^{\top}\overline{\sigma}^1\right)(\widetilde{S}^{\varepsilon,x}_s,\overline{S}^{\varepsilon,x}_s)\right)\right]ds\\
& \ \ \ +{2}\int^t_0e^{\frac{\eta}{\varepsilon}s}\left[\left\langle\overline{S}^{\varepsilon,x}_s,\overline{b}^2(\widetilde{S}^{\varepsilon,x}_s,\overline{S}^{\varepsilon,x}_s)\right\rangle+ G\left(\left(\overline{\sigma}^1(\widetilde{S}^{\varepsilon,x}_s,\overline{S}^{\varepsilon,x}_s)\right)^{\top}\overline{\sigma}^1(\widetilde{S}^{\varepsilon,x}_s,\overline{S}^{\varepsilon,x}_s)\right)\right]ds,
\end{split}
\end{align}
where  $M^1_t$  is a symmetric $G$-martingale with $M^{1}_0=0$.

Recalling assumptions (H1) and (H2), we could get that
\begin{align*}
&2\left\langle\overline{S}^{\varepsilon,x}_s,\overline{b}(\widetilde{S}^{\varepsilon,x}_s,\overline{S}^{\varepsilon,x}_s)\right\rangle+2G\left(\left(\overline{\sigma}(\widetilde{S}^{\varepsilon,x}_s,\overline{S}^{\varepsilon,x}_s)\right)^{\top}\left(\overline{\sigma}(\widetilde{S}^{\varepsilon,x}_s,\overline{S}^{\varepsilon,x}_s)\right)\right)\\
&\leq -\frac{3}{2}\eta\left|\overline{S}^{\varepsilon,x}_s\right|^2+C(L_1,\eta)\left(1+\left|\widetilde{S}^{\varepsilon,x}_s\right|^2\right),
\end{align*}
and
\begin{align*}
&2\left\langle\overline{S}^{\varepsilon,x}_s,\overline{b}^1(\widetilde{S}^{\varepsilon,x}_s,\overline{S}^{\varepsilon,x}_s)\right\rangle +2G\left(\left(\left(\overline{\sigma}^1\right)^{\top}\overline{\sigma}+\left(\overline{\sigma}\right)^{\top}\overline{\sigma}^1\right)(\widetilde{S}^{\varepsilon,x}_s,\overline{S}^{\varepsilon,x}_s)\right)\leq C(L_1)\left(1+\left|\overline{S}^{\varepsilon,x}_s\right|^2+\left|\widetilde{S}^{\varepsilon,x}_s\right|^2\right),\\
&2\left\langle\overline{S}^{\varepsilon,x}_s,\overline{b}^2(\widetilde{S}^{\varepsilon,x}_s,\overline{S}^{\varepsilon,x}_s)\right\rangle+ 2G\left(\left(\overline{\sigma}^1(\widetilde{S}^{\varepsilon,x}_s,\overline{S}^{\varepsilon,x}_s)\right)^{\top}\overline{\sigma}^1(\widetilde{S}^{\varepsilon,x}_s,\overline{S}^{\varepsilon,x}_s)\right) \leq C(L_1)\left(1+\left|\overline{S}^{\varepsilon,x}_s\right|^2+\left|\widetilde{S}^{\varepsilon,x}_s\right|^2\right).
\end{align*}
It follows from inequality \eqref{myw60578}
 that
 \begin{align*}\begin{split}
 &e^{\frac{\eta}{\varepsilon}t}\left|\overline{S}^{\varepsilon,x}_t\right|^2\leq |\bar{x}|^2+M^1_t+\frac{1}{2\varepsilon}\left(C(L_1)(\sqrt{\varepsilon}+\varepsilon)-\eta\right)\int^t_0e^{\frac{\eta}{\varepsilon}s}\left|\overline{S}^{\varepsilon,x}_s\right|^2ds+\frac{C(L_1,\eta)}{\varepsilon}\int^t_0e^{\frac{\eta}{\varepsilon}s}\left(1+\left|\widetilde{S}^{\varepsilon,x}_s\right|^2\right)ds,
 \end{split}
 \end{align*}
 Thus, if $\varepsilon\leq \frac{\eta^2}{C(L_1)}$, we could obtain that
 \begin{align*}
\mathbb{\hat{E}}\left[\left|\overline{S}^{\varepsilon,x}_t\right|^2\right]\leq e^{-\frac{\eta}{\varepsilon}t}|\bar{x}|^2+\frac{C(L_1,\eta)}{\varepsilon}\int^t_0e^{\frac{\eta}{\varepsilon}(s-t)}\left(1+\mathbb{\hat{E}}\left[\left|\widetilde{S}^{\varepsilon,x}_s\right|^2\right]\right)ds, \ \forall t\in[0,T].
\end{align*}
Following the  proof of Assertion (iii) in Lemma \ref{myw603}, we could get
\[
\mathbb{\hat{E}}\left[\sup\limits_{0\leq t\leq T}\left|\widetilde{S}^{\varepsilon,x}_t\right|^2\right]+\sup\limits_{0\leq t\leq T}\mathbb{\hat{E}}\left[\left|\overline{S}^{\varepsilon,x}_t\right|^2\right]\leq  C(L_1,\eta,T)\left(1+|\tilde{x}|^2+|\bar{x}|^2\right).
\]
On the other hand, if $\varepsilon> \frac{\eta^2}{C(L_1)}$, then all Lipschitz constants are uniformly bounded, and the result is trivial.

{\bf Step 2} (Assertion {(ii)}).
Applying $G$-It\^{o}'s formula  yields that
\begin{align*}\begin{split}
&e^{\frac{\eta}{\varepsilon}t}\left|\overline{S}^{\varepsilon,x}_t-\overline{X}^{\varepsilon,x}_t\right|^2-M_t^2-\frac{\eta}{\varepsilon}\int^t_0e^{\frac{\eta}{\varepsilon}s}\left|\overline{S}^{\varepsilon,x}_s-\overline{X}^{\varepsilon,x}_s\right|^2ds\\
&\leq
\frac{2}{\varepsilon}\int^t_0e^{\frac{\eta}{\varepsilon}s}\left[\left\langle\overline{S}^{\varepsilon,x}_s-\overline{X}^{\varepsilon,x}_s,\overline{b}(\widetilde{S}^{\varepsilon,x}_s,\overline{S}^{\varepsilon,x}_s)-\overline{b}(\widetilde{X}^{\varepsilon,x}_s,\overline{X}^{\varepsilon,x}_s)\right\rangle  +G\left(\Gamma_s^{\top}\Gamma_s\right)\right]ds\\
&\ \ +\frac{2}{\sqrt{\varepsilon}}\int^t_0e^{\frac{\eta}{\varepsilon}s}\left[\left\langle\overline{S}^{\varepsilon,x}_s-\overline{X}^{\varepsilon,x}_s,\overline{b}^1(\widetilde{S}^{\varepsilon,x}_s,\overline{S}^{\varepsilon,x}_s)\right\rangle+G\left(\Gamma_s^{\top}\overline{\sigma}^1(\widetilde{S}^{\varepsilon,x}_s,\overline{S}^{\varepsilon,x}_s)+\left(\overline{\sigma}^1(\widetilde{S}^{\varepsilon,x}_s,\overline{S}^{\varepsilon,x}_s)\right)^{\top}\Gamma_s\right)\right] ds\\
&\ \ +{2}\int^t_0e^{\frac{\eta}{\varepsilon}s}\left[\left\langle\overline{S}^{\varepsilon,x}_s-\overline{X}^{\varepsilon,x}_s,\overline{b}^2(\widetilde{S}^{\varepsilon,x}_s,\overline{S}^{\varepsilon,x}_s)\right\rangle+G\left(\left(\overline{\sigma}^1(\widetilde{S}^{\varepsilon,x}_s,\overline{S}^{\varepsilon,x}_s)\right)^{\top}\overline{\sigma}^1(\widetilde{S}^{\varepsilon,x}_s,\overline{S}^{\varepsilon,x}_s)\right)\right] ds,
\end{split}
\end{align*}
where  $M_t^2$  is a symmetric $G$-martingale and $\Gamma_t=\overline{\sigma}(\widetilde{S}^{\varepsilon,x}_t,\overline{S}^{\varepsilon,x}_t)-\overline{\sigma}(\widetilde{X}^{\varepsilon,x}_t,\overline{X}^{\varepsilon,x}_t)$.

According to assumptions (H1) and (H2), we derive  that
\begin{align*}
&2\left\langle\overline{S}^{\varepsilon,x}_s-\overline{X}^{\varepsilon,x}_s,\overline{b}(\widetilde{S}^{\varepsilon,x}_s,\overline{S}^{\varepsilon,x}_s)-\overline{b}(\widetilde{X}^{\varepsilon,x}_s,\overline{X}^{\varepsilon,x}_s)\right\rangle  +2G\left(\Gamma_s^{\top}\Gamma_s\right)\\
& \leq -\frac{3}{2}\eta\left|\overline{S}^{\varepsilon,x}_s-\overline{X}^{\varepsilon,x}_s\right|^2+ C(L_1,\eta)\left|\widetilde{S}^{\varepsilon,x}_s-\widetilde{X}^{\varepsilon,x}_s\right|^2,
\end{align*}
and
\begin{align*}
&2\left\langle\overline{S}^{\varepsilon,x}_s-\overline{X}^{\varepsilon,x}_s,\overline{b}^1(\widetilde{S}^{\varepsilon,x}_s,\overline{S}^{\varepsilon,x}_s)\right\rangle+2G\left(\Gamma_s^{\top}\overline{\sigma}^1(\widetilde{S}^{\varepsilon,x}_s,\overline{S}^{\varepsilon,x}_s)+\left(\overline{\sigma}^1(\widetilde{S}^{\varepsilon,x}_s,\overline{S}^{\varepsilon,x}_s)\right)^{\top}\Gamma_s\right)\\
&\leq \frac{\eta}{4\sqrt{\varepsilon}}\left|\overline{S}^{\varepsilon,x}_s-\overline{X}^{\varepsilon,x}_s\right|^2+\frac{\eta}{4\sqrt{\varepsilon}}\left|\widetilde{S}^{\varepsilon,x}_s-\widetilde{X}^{\varepsilon,x}_s\right|^2+C(L_1,\eta)\sqrt{\varepsilon}\left(1+\left|\widetilde{S}^{\varepsilon,x}_s\right|^2+\left|\overline{S}^{\varepsilon,x}_s\right|^2 \right),\\
&2\left\langle\overline{S}^{\varepsilon,x}_s-\overline{X}^{\varepsilon,x}_s,\overline{b}^2(\widetilde{S}^{\varepsilon,x}_s,\overline{S}^{\varepsilon,x}_s)\right\rangle+2G\left(\left(\overline{\sigma}^1(\widetilde{S}^{\varepsilon,x}_s,\overline{S}^{\varepsilon,x}_s)\right)^{\top}\overline{\sigma}^1(\widetilde{S}^{\varepsilon,x}_s,\overline{S}^{\varepsilon,x}_s)\right)\\
& \leq \frac{\eta}{4\varepsilon}\left|\overline{S}^{\varepsilon,x}_s-\overline{X}^{\varepsilon,x}_s\right|^2+C(L_1,\eta)\left(1+\left|\widetilde{S}^{\varepsilon,x}_s\right|^2+\left|\overline{S}^{\varepsilon,x}_s\right|^2 \right).
\end{align*}
Putting the above four inequalities together, we get that
\begin{align*}
e^{\frac{\eta}{\varepsilon}t}\left|\overline{S}^{\varepsilon,x}_t-\overline{X}^{\varepsilon,x}_t\right|^2\leq & M_t^2+\frac{C(L_1,\eta)}{\varepsilon}\int^t_0e^{\frac{\eta}{\varepsilon}s}\left|\widetilde{S}^{\varepsilon,x}_s-\widetilde{X}^{\varepsilon,x}_s\right|^2ds\\
&\ \ +C(L_1,\eta)\int^t_0e^{\frac{\eta}{\varepsilon}s} \left(1+\left|\widetilde{S}^{\varepsilon,x}_s\right|^2+\left|\overline{S}^{\varepsilon,x}_s\right|^2\right)ds.
\end{align*}
Taking $G$-expectation on both sides and recalling Assertion (i), we obtain that for each $t\in[0,T]$,
\begin{align*}
\mathbb{\hat{E}}\left[\left|\overline{S}^{\varepsilon,x}_t-\overline{X}^{\varepsilon,x}_t\right|^2\right]\leq
C(L_1,\eta,T)(1+|x|^2)\varepsilon+ \frac{C(L_1,\eta)}{\varepsilon}\int^t_0e^{\frac{\eta}{\varepsilon}(s-t)}\mathbb{\hat{E}}\left[\left|\widetilde{S}^{\varepsilon,x}_s-\widetilde{X}^{\varepsilon,x}_s\right|^2\right]ds,
\end{align*}
which indicates that
\begin{align*}
\int^t_0\mathbb{\hat{E}}\left[\left|\overline{S}^{\varepsilon,x}_s-\overline{X}^{\varepsilon,x}_s\right|^2\right]ds
\leq C(L_1,\eta,T)(1+|x|^2)\varepsilon+C(L_1,\eta)\int^t_0\mathbb{\hat{E}}\left[\sup\limits_{0\leq r\leq s }\left|\widetilde{S}^{\varepsilon,x}_r-\widetilde{X}^{\varepsilon,x}_r\right|^2\right]ds.
\end{align*}
It follows from H\"{o}lder's inequality and BDG's inequality that
\begin{align*}
\mathbb{\hat{E}}\left[\sup\limits_{0\leq s\leq t}\left|\widetilde{S}^{\varepsilon,x}_s-\widetilde{X}^{\varepsilon,x}_s\right|^2\right]\leq C(L_1,\eta,T)\left((1+|x|^2)\varepsilon+\int^t_0\mathbb{\hat{E}}\left[\sup\limits_{0\leq r\leq s}\left|\widetilde{S}^{\varepsilon,x}_r-\widetilde{X}^{\varepsilon,x}_r\right|^2\right]ds\right),
\end{align*}
which together with Gronwall's inequality implies that
\begin{align*}
\mathbb{\hat{E}}\left[\sup\limits_{0\leq s\leq T}\left|\widetilde{S}^{\varepsilon,x}_s-\widetilde{X}^{\varepsilon,x}_s\right|^2\right]
\leq C(L_1,\eta,T)(1+|x|^2)\varepsilon.
\end{align*}
The proof is complete.
\end{proof}

Then we have the following asymptotics result.
\begin{theorem}\label{myw2020615}
Suppose assumptions  \emph{(H1)}-\emph{(H4)} hold. Then, for each $\varphi\in C(\mathbb{R}^n)$ of  polynomial growth,
\[
\lim\limits_{\varepsilon\rightarrow0}\mathbb{\hat{E}}\left[\varphi(\widetilde{S}^{\varepsilon,x}_t)\right]=\widetilde{u}(t,\tilde{x}), \ \forall \ (t,\tilde{x},\bar{x})\in [0,\infty)\times\mathbb{R}^{2n},
\]
where $\widetilde{u}$ is the unique viscosity solution  to the averaged PDE \eqref{PDE511}   satisfying the polynomial growth condition.
\end{theorem}
\begin{proof}
It suffices to prove that $\lim\limits_{\varepsilon\rightarrow0}\mathbb{\hat{E}}\left[\left|\varphi(\widetilde{S}^{\varepsilon,x}_t)-\varphi(\widetilde{X}^{\varepsilon,x}_t)\right|\right]=0.$
Note that $\varphi$ satisfies  the polynomial growth condition.
According to Lemma \ref{myw6101}, we obtain that for  each $N>0$,
\begin{align*}
\mathbb{\hat{E}}\left[\left|\varphi(\widetilde{S}^{\varepsilon,x}_t)-\varphi(\widetilde{X}^{\varepsilon,x}_t)\right|\right]
\leq \mathbb{\hat{E}}\left[\left|\varphi(\widetilde{S}^{\varepsilon,x}_t)-\varphi(\widetilde{X}^{\varepsilon,x}_t)\right|I_{|\widetilde{S}^{\varepsilon,x}_t|\leq N}I_{|\widetilde{X}^{\varepsilon,x}_t|\leq N}\right]+C(L_1,L_2,T,\varphi)\frac{1+|\tilde{x}|^{C(\varphi)}}{N}.
\end{align*}
For each $\epsilon>0$, there is a constant $\delta>0$ such that
$
|\varphi(x)-\varphi(y)|\leq \epsilon
$
whenever $|x|\leq N$, $|y|\leq N$, and $|x-y|\leq \delta$.
It follows that
\begin{align*}
&\mathbb{\hat{E}}\left[\left|\varphi(\widetilde{S}^{\varepsilon,x}_t)-\varphi(\widetilde{X}^{\varepsilon,x}_t)\right|I_{|\widetilde{S}^{\varepsilon,x}_t|\leq N}I_{|\widetilde{X}^{\varepsilon,x}_t|\leq N}\right]  \leq \epsilon+
\mathbb{\hat{E}}\left[\left|\varphi(\widetilde{S}^{\varepsilon,x}_t)-\varphi(\widetilde{X}^{\varepsilon,x}_t)\right|I_{|\widetilde{S}^{\varepsilon,x}_t-\widetilde{X}^{\varepsilon,x}_t|\geq \delta}\right]\\
&  \leq \epsilon+
\mathbb{\hat{E}}\left[\left|\varphi(\widetilde{S}^{\varepsilon,x}_t)-\varphi(\widetilde{X}^{\varepsilon,x}_t)\right|^2\right]\mathbb{\hat{E}}\left[\left|\widetilde{S}^{\varepsilon,x}_t-\widetilde{X}^{\varepsilon,x}_t\right|^2\right]\delta^{-2}
\leq \epsilon+C(L_1,L_2,\eta,T,\varphi)\left(1+|x|^{C(\varphi)}\right)\frac{\varepsilon}{\delta^2},
\end{align*}
where we have used Lemma \ref{myw60313} in the last inequality.
As a result, we deduce that
\begin{align*}
\limsup\limits_{\varepsilon\rightarrow 0}\mathbb{\hat{E}}\left[\left|\varphi(\widetilde{S}^{\varepsilon,x}_t)-\varphi(\widetilde{X}^{\varepsilon,x}_t)\right|\right]
\leq \epsilon+C(\varphi)\frac{1+|\tilde{x}|^{C(\varphi)}}{N}.
\end{align*}
Sending $\epsilon\rightarrow 0$, and then letting $N\rightarrow \infty$, we could complete the proof.
\end{proof}

\section*{Acknowledgements}

The authors would like to thank the editor and the anonymous
referees for their valuable suggestions and comments which led to a much improved version of the manuscript.

\appendix
\renewcommand\thesection{\normalsize Appendix A:  Ergodic theory in the $G$-expectation framework}
\section{ }

\renewcommand\thesection{A}
\normalsize

In what follows, we shall recall some basic results about  ergodic $G$-BSDE: for each $x\in\mathbb{R}^n$ and for any $ 0\leq t\leq T<\infty$,
\begin{align}\label{AppHM}\begin{cases}
&{\displaystyle  X^x_t=x+\int^t_0b(X^x_s)ds+\sum\limits_{i,j=1}^d\int^t_0h_{ij}(X^x_s)d\langle B^i,B^j\rangle_s+\int^t_0\sigma(X^x_s)dB_s,}\\
&{\displaystyle {Y}_{t}^{x}={Y}^x_T+\int_{t}^{T}\left(g(X_{s}^{x})-\lambda\right)ds+\sum_{i,j=1}^d\int^T_tg^{\prime}_{ij}(X^x_s)d\langle B^i,B^j\rangle_s-\int_{t}^{T}{Z}_{s}^{x}dB_{s}-({K}_{T}^{x}-{K}_{t}^{x}),}
\end{cases}
\end{align}
where $b,h_{ij}=h_{ji}:\mathbb{R}^n\rightarrow\mathbb{R}^n$, $\sigma:\mathbb{R}^n\rightarrow\mathbb{R}^{n\times d}$ and $g,g^{\prime}_{ij}:\mathbb{R}^n\rightarrow\mathbb{R}$ are  deterministic functions satisfying the following conditions.
\begin{description}
 \item[(A1)] There exist two constant $\kappa_1>0$ and $\kappa_2>0$ such that for each $x,x^{\prime}\in\mathbb{R}^n$,
\begin{align*}
|\ell(x)-\ell(x^{\prime})|\leq
\kappa_1|x-x^{\prime}|,\ \text{for $\ell=b,h_{ij},\sigma$}, \ \text{and} \ |\ell(x)-\ell(x^{\prime})|\leq
\kappa_2|x-x^{\prime}|,\ \text{for $\ell=g,g^{\prime}$}.
\end{align*}
 \item[(A2)] There exists a constant $\eta>0$ such that for each $x,x^{\prime}\in\mathbb{R}^n$,
\begin{align*}
&G\left((\sigma({x})-\sigma({x}^{\prime}%
))^{\top}(\sigma({x})-\sigma({x}^{\prime}))+2\left[\langle {x}-{x}^{\prime}%
,h_{ij}({x})-h_{ij}({x}^{\prime})\rangle\right]_{i,j=1}^{d}\right)+\langle {x}-{x}^{\prime
},b({x})-b({x}^{\prime})\rangle \\ & \leq-\eta|x-x^{\prime}|^{2}.
\end{align*}
\end{description}

\begin{lemma}\label{myw2}
 Under assumptions \emph{(A1)} and \emph{(A2)},  there exists a constant ${C(\kappa_1,\eta)}$, such that for any ${x},{x}^{\prime}\in\mathbb{R}^n$ and $t,s\geq 0$,
\begin{description}
\item[(i)] $\hat{\mathbb{E}}[|{X}_{t}^{{x}}|^{2}]\leq {C(\kappa_1,\eta)}(1+|{x}|^2+|\bar{\kappa}|^2)$,
\item[(ii)] $\hat{\mathbb{E}}[|{X}_{t}^{{x}}-{X}_{t}^{{x}^{\prime}}|^{2}]\leq\exp(-2\eta t)|{x}-{x}^{\prime}|^{2}$,
\end{description}
where  $\bar{\kappa}:=\max\{|b(0)|,|\sigma(0)|,|h_{ij}(0)|, 1\leq i,j\leq d\}.$
\end{lemma}
\begin{proof}
The proof is immediate from Lemma 3.2 of \cite{HW} or Lemma 4.1 of \cite{HW1}.
\end{proof}

Let $S_{G}^{0}(0,T)=\{h(t,B_{t_{1}\wedge t},\cdot\cdot\cdot,B_{t_{n}\wedge
t}):t_{1},\ldots,t_{n}\in\lbrack0,T],h\in C_{b.lip}(\mathbb{R}^{n\times d+1})\}$. For
$\eta\in S_{G}^{0}(0,T)$, set $\Vert\eta\Vert_{S_{G}^{2}}=\{
\mathbb{\hat{E}}[\sup_{t\in\lbrack0,T]}|\eta_{t}|^{2}]\}^{\frac{1}{2}}$.
Denote by $S_{G}^{2}(0,T)$ the completion of $S_{G}^{0}(0,T)$ under the norm
$\Vert\cdot\Vert_{S_{G}^{2}}$ (see \cite{Liu1} for more related research). For the sake of brevity, denote by $\mathfrak{S}_{G}^{2}(0,\infty)$  the collection of processes $(Y_t,Z_t,K_t)_{t\geq 0}$ such that, for each $T>0$, $(Y_t)_{t\in[0,T]}\in S_G^{2}(0,T)$, $(Z_t)_{t\in[0,T]}\in M_G^{2}(0,T;\mathbb{R}^d)$ and $(K_t)_{t\in[0,T]}\in S_G^{2}(0,T)$ is a continuous non-increasing $G$-martingale starting from origin.

\begin{lemma}\label{myw6}
Suppose that \emph{(A1)} and \emph{(A2)} hold. Then, the
$G$-EBSDE \eqref{AppHM} has a solution $(Y^x,Z^x,K^x,\lambda)\in\mathfrak{S}_{G}^{2}(0,\infty)\times\mathbb{R}$  for each $x\in\mathbb{R}^n$, where the constant $\lambda$ is independent of the argument $x$.
Moreover, there exists a  continuous function $v$  satisfying
\[
v(0)=0, \ |v(x)-v(x^{\prime})|\leq C(\eta)\kappa_2|x-x^{\prime}|, \ \forall x,x^{\prime}\in\mathbb{R}^n,
\]
such that $Y_t^x=v(X^x_t)$ for each $(t,x)\in[0,\infty)\times\mathbb{R}^n$.
\end{lemma}
\begin{proof}
We shall  give the sketch of the proof for readers' convenience. Without loss of generality, assume $g^{\prime}_{ij}=0$.
For each $\epsilon>0$, consider the following $G$-BSDE with infinite horizon:
\[{Y}_{t}^{\epsilon,x}={Y}^{\epsilon,x}_T+\int_{t}^{T}\left(g(X_{s}^{x})-\epsilon Y^{\epsilon,x}_s\right)ds-\int_{t}^{T}{Z}_{s}^{\epsilon,x}dB_{s}-({K}_{T}^{\epsilon,x}-{K}_{t}^{\epsilon,x}), \ \forall 0\leq t\leq T<\infty.
\]
We define the function $v^{\epsilon}(x):={Y}_{0}^{\epsilon,x}$ for each $x\in\mathbb{R}^n$. Then from Lemma \ref{myw2}, the proof of Theorem 3.1 and Lemma 4.2 of \cite{HW1}, we have that
\[
|v^{\epsilon}(x)|\leq  C(\kappa_1,\kappa_2,\eta)\frac{1+|x|+|\bar{\kappa}|}{\epsilon}\ \text{and}\ |v^{\epsilon}(x)-v^{\epsilon}(x^{\prime})|\leq C(\eta)\kappa_2|x-x^{\prime}|.
\]
Denote $\overline{v}^{\epsilon}(x)={v}^{\epsilon}(x)-{v}^{\epsilon}(0)$.
Note that $\overline{v}^{\epsilon}(x)$ is a  uniformly Lipschitz function.
Thus, by a diagonal procedure, we can construct  a
sequence $\epsilon_m\downarrow 0$ such that $\overline{v}^{\epsilon_m}(x)\rightarrow v(x)$ for all $x\in\mathbb{R}^n$ and $\epsilon_m{v}^{\epsilon_m}(0)\rightarrow \lambda$ for some real number $\lambda$.
Finally, by a similar analysis as in Theorem 5.1 of \cite{HW1}, we can get the desired result.
\end{proof}

Then, we have the following asymptotic property, which can be seen as the ergodic theorem in the $G$-expectation framework.
\begin{lemma}\label{myw7}
Assume  conditions \emph{(A1)} and \emph{(A2)} hold. Then, for each $T\in[0,\infty)$, we have
\[
\left|\mathbb{\hat{E}}\left[{\int^{T}_0g(X^x_s)ds}+\sum_{i,j=1}^d\int^{T}_0g^{\prime}_{ij}(X^x_s)d\langle B^i,B^j\rangle_s\right]-\lambda T\right|\leq {C(\kappa_1,\eta)\kappa_2(1+|x|+|\bar{\kappa}|)}, \ \forall x\in\mathbb{R}^n.
\]
where $\bar{\kappa}$ and $\lambda$ are given by Lemma \emph{\ref{myw2}} and Lemma \emph{\ref{myw6}}, respectively. In particular,\[
\lambda=\lim\limits_{T\rightarrow\infty}\frac{1}{T}\mathbb{\hat{E}}\left[{\int^{T}_0g(X^x_s)ds}+\sum_{i,j=1}^d\int^{T}_0g^{\prime}_{ij}(X^x_s)d\langle B^i,B^j\rangle_s\right].
\]
\end{lemma}
\begin{proof}
From equation \eqref{AppHM},   we get that
\[
{Y}_{0}^{x}=\mathbb{\hat{E}}\left[{Y}^x_{T}+\int_{0}^{T}g(X_{s}^{x})ds+\sum_{i,j=1}^d\int^{T}_0g^{\prime}_{ij}(X^x_s)d\langle B^i,B^j\rangle_s-\lambda T\right],
\]
which implies that
 \begin{align*}
 &\left|\mathbb{\hat{E}}\left[\int_{0}^{T}g(X_{s}^{x})ds+\sum_{i,j=1}^d\int^{T}_0g^{\prime}_{ij}(X^x_s)d\langle B^i,B^j\rangle_s\right]-\lambda T\right|\\& \leq   \left|\mathbb{\hat{E}}\left[{Y}^x_{T}+\int_{0}^{T}g(X_{s}^{x})ds+\sum_{i,j=1}^d\int^{T}_0g^{\prime}_{ij}(X^x_s)d\langle B^i,B^j\rangle_s\right]-\lambda T\right|+{\mathbb{\hat{E}}[|Y^x_{T}|]}\leq  {|Y^x_0|+\mathbb{\hat{E}}[|Y^x_{T}|]}.
\end{align*}
Recalling Lemma \ref{myw6}, there exists a constant $C(\eta)$ such that
\[
|Y^x_s|\leq C(\eta)\kappa_2|X^x_s|, \ \forall s\geq 0,
\]
which together with Assertion (i) of Lemma  \ref{myw2},  indicates that
\[
\mathbb{\hat{E}}\left[|{Y}^x_{T}|\right]\leq  C(\eta)\kappa_2\mathbb{\hat{E}}[|X^x_{T}|]\leq  C(\kappa_1,\eta)\kappa_2(1+|x|+|\bar{\kappa}|).
\]
It follows that
\[
 \left|\mathbb{\hat{E}}\left[{\int^{T}_0g(X^x_s)ds}+\sum_{i,j=1}^d\int^{T}_0g^{\prime}_{ij}(X^x_s)d\langle B^i,B^j\rangle_s\right]-\lambda T\right|\leq C(\kappa_1,\eta)\kappa_2(1+|x|+|\bar{\kappa}|),
\]
which completes the proof.
\end{proof}

\appendix
\renewcommand\thesection{\normalsize Appendix B: Comparison theorem for the averaged PDE}
\section{ }
\renewcommand\thesection{B}
In this appendix, we shall state the comparison theorem for PDE \eqref{PDE511}.

\begin{lemma} \label{myw507} Let $\widetilde{v}^1$ be a viscosity subsolution  and $\widetilde{v}^2$ be a viscosity supersolution to PDE \eqref{PDE511} satisfying the polynomial growth condition, respectively. Then
$\widetilde{v}^1\leq \widetilde{v}^2$ on $[0, T]\times\mathbb{R}^{n}$  provided that $\widetilde{v}^1|_{t=0}\leq \widetilde{v}^2|_{t=0}$.
\end{lemma}

\begin{proof}
The main idea is from Theorem 8.6 in \cite{KKPPQ} and Theorem 2.2 in Appendix C of \cite{P10}. For reader's convenience, we shall give the sketch of the proof.

 For some constant $\lambda>0 $  to be chosen below, we set $\xi(\tilde{x}):=(1+|\tilde{x}|^{2})^{l/2}$ and
\begin{equation*}
\widetilde{v}_1(t,\tilde{x}):=\widetilde{v}^1(t,\tilde{x})\xi^{-1} (\tilde{x})e^{-\lambda t}, \ \widetilde{v}_2(t,\tilde{x}):=-\widetilde{v}^2(t,\tilde{x})\xi^{-1} (\tilde{x})e^{-\lambda t},
\end{equation*}%
where $l\geq 2$ is chosen to be large enough such that $|\widetilde{v}%
_{i}|\rightarrow 0$ uniformly  as $x\rightarrow\infty$. It
is easy to check that, $\widetilde{v}_{i}$
is a bounded viscosity subsolution   of
\begin{equation*}
\partial _{t}\widetilde{v}_{i}+\lambda\widetilde{v}_i-\widetilde{G}^*_i(\tilde{x},\widetilde{v}%
_{i},D\widetilde{v}_i,D^{2}\widetilde{v}_{i})=0,
\end{equation*}%
where the function
$\widetilde{G}^*_1(\tilde{x},v,p,X)=\widetilde{G}^*(\tilde{x},v,p,X),\widetilde{G}^*_2(\tilde{x},v,p,X)=-\widetilde{G}^*(\tilde{x},-v,-p,-X)$ and
\begin{align}\label{myw510}\begin{split}
\widetilde{G}^*(\tilde{x},v,p,X):&=e^{-\lambda t}\xi ^{-1}\widetilde{G}(\tilde{x},e^{\lambda
t}(p\xi(\tilde{x})+vD\xi (\tilde{x}) ),e^{\lambda t}(X\xi(\tilde{x}) +p\otimes D\xi (\tilde{x}) +D\xi (\tilde{x}) \otimes
p+vD^2\xi (\tilde{x}) ))\\
&=\widetilde{G}(\tilde{x},p+v\eta(\tilde{x}),X +p\otimes \eta(\tilde{x}) +\eta(\tilde{x}) \otimes
p+v\kappa(\tilde{x}))
\end{split}
\end{align}
for any $(\tilde{x},v,p,X)\in\mathbb{R}^n\times\mathbb{R}\times\mathbb{R}^n\times\mathbb{S}(n)$.
Here $p\otimes \eta(\tilde{x})=[p^i\eta^j(\tilde{x})]_{i,j}$ and
\begin{align*}
\eta(\tilde{x}):=\xi^{-1}(\tilde{x})D\xi (\tilde{x})& =l (1+|\tilde{x}|^{2})^{-1}\tilde{x},\ \ \ \  \\
\kappa(\tilde{x}):=\xi^{-1}(\tilde{x})D^{2}\xi (\tilde{x})& =l(1+|\tilde{x}|^{2})^{-1}I_{n}+l(l-2)(1+|\tilde{x}|^{2})^{-2}\tilde{x}\otimes \tilde{x}.
\end{align*}%
Note that $l\geq 2$, $\eta$ and $\kappa$ are uniformly bounded functions. Then, using Assertion (iv) of Lemma \ref{myw502}, we could choose $\lambda$ large enough, so that the function \begin{align}\label{myw504}
v\rightarrow-\lambda v+\widetilde{G}^*(\tilde{x},v,p,X)\ \text{is non-increasing   for any  $(\tilde{x},p,X)\in\mathbb{R}^n\times\mathbb{R}^n\times\mathbb{S}(n)$.}\end{align}

Next, we shall verify that $\widetilde{G}^*$ satisfies the  regularity  condition (3.14)  in \cite{CMI} for the comparison principle.
Suppose that  $A, B\in\mathbb{S}(n)$ satisfies
\begin{equation*}
 \left(
\begin{array}
[c]{cccc}%
A  & 0\\
0  & B
\end{array}
\right)  \leq 3\alpha\left(
\begin{array}
[c]{cccc}%
I_n & -I_n\\
-I_n  & I_n
\end{array}\right),
\end{equation*}
for some $\alpha>0$.
Note that $G(A)\leq\frac{1}{2}\overline{\sigma}^2\mathrm{tr}[A]$ for any $A\geq 0$. Then, we get that
\begin{align*}
&\mathbb{\hat{E}}\left[\sum_{i,j=1}^d\int^T_0 \widetilde{\sigma}^A_{ij}(\tilde{x},\overline{X}^{(\tilde{x},\bar{x})}_s)d\langle B^i, B^j\rangle_s-\sum_{i,j=1}^d\int^T_0 \widetilde{\sigma}^{-B}_{ij}(\tilde{y},\overline{X}^{(\tilde{y},\bar{x})}_s)d\langle B^i, B^j\rangle_s\right]
\\
&\leq2 \mathbb{\hat{E}}\left[\int^T_0 G\left(\widetilde{\sigma}^{\top}(\tilde{x},\overline{X}^{(\tilde{x},\bar{x})}_s) A \widetilde{\sigma}(\tilde{x},\overline{X}^{(\tilde{x},\bar{x})}_s)+\widetilde{\sigma}^{\top}(\tilde{y},\overline{X}^{(\tilde{y},\bar{x})}_s) B \widetilde{\sigma}(\tilde{y},\overline{X}^{(\tilde{y},\bar{x})}_s)\right )ds \right]\\
&\leq 3\overline{\sigma}^2\alpha\mathbb{\hat{E}}\left[\int^T_0\left|\widetilde{\sigma}(\tilde{x},\overline{X}^{(\tilde{x},\bar{x})}_s)-\widetilde{\sigma}(\tilde{y},\overline{X}^{(\tilde{y},\bar{x})}_s)\right|^2 ds \right] \leq C(L_1,\eta)\alpha |\tilde{x}-\tilde{y}|^2 T,
\end{align*}
where we have used assumption (H1) and estimate \eqref{myw503} in the last inequality.
Thus, from the definitions of $\widetilde{G},\widetilde{G}^*$ and by a similar analysis as the proof of Assertion (iv) of  Lemma \ref{myw502} , we conclude that
\begin{align}\label{myw505}
\widetilde{G}^*(\tilde{x},v,\alpha(\tilde{x}-\tilde{y}),A)-\widetilde{G}^*(\tilde{y},v,\alpha(\tilde{x}-\tilde{y}),-B)\leq C(L_1,L_2,\eta)(1+|\tilde{x}|^2+|\tilde{y}|^2)(|v||\tilde{x}-\tilde{y}|+\alpha |\tilde{x}-\tilde{y}|^2).
\end{align}

Finally, we will prove $\widetilde{v}_{1}+\widetilde{v}_{2}\leq 0$. By the proof of Theorem 2.2 in \cite{P10},
it suffices to prove the result under the  additional assumptions: for each $\bar{\delta}>0$,
\begin{equation}%
\partial_{t}\widetilde{v}_{i}+\lambda\widetilde{v}_{i} -\widetilde{G}^*_{i}(\tilde{x},\widetilde{v}_{i},D\widetilde{v}_{i},D^{2}\widetilde{v}_{i})\leq-\bar
{\delta}/T^{2},\ \text{and }\lim_{t\rightarrow
T}\widetilde{v}_{i}(t,\tilde{x})=-\infty \  \text{uniformly on }
\mathbb{R}^{n}.
\label{ineq-c}%
\end{equation}
Assume the contrary that
\[
\sup_{(t,x)\in \lbrack0,T)\times \mathbb{R}^{n}}(\widetilde{v}_{1}(t,x)+\widetilde{v}_{2}(t,x))>0.
\]
Note that $(\widetilde{v}_1(t,\tilde{x}))^++(\widetilde{v}_2(t,\tilde{x}))^+\rightarrow 0$ uniformly as $\tilde{x}\rightarrow\infty$.
Thus, taking $\beta_1=\beta_2=1$ and following the proof of  Theorem 2.2 in \cite{P10} line by line,
 for large enough $\alpha>0$, we could find some point $(t^{\alpha},\tilde{x}_{1}^{\alpha},\tilde{x}_{2}^{\alpha})$ inside a compact subset of $[0,T)\times\mathbb{R}^{2n}$, so that $\widetilde{v}_{1}(t^{\alpha},\tilde{x}_{1}^{\alpha})+\widetilde{v}_{2}(t^{\alpha},\tilde{x}_{2}^{\alpha})-\frac{\alpha}{2}|\tilde{x}_{2}^{\alpha}-\tilde{x}_{1}^{\alpha}|^2>0$ and
\[
\  \lim_{\alpha\rightarrow\infty}\alpha|\tilde{x}_{1}^{\alpha}-\tilde{x}_{2}^{\alpha}|^2=0,
 \ \text{and}\ \lim_{\alpha\rightarrow\infty}(t^\alpha,\tilde{x}_{1}^{\alpha},\tilde{x}_{2}^{\alpha})=({t}^*,\tilde{x}^*,\tilde{x}^*)\ \text{for some $(t^*,\tilde{x}^*)\in(0,T)\times\mathbb{R}^n$}.
\]
Then, there exist
$b_{i}^{\alpha}\in \mathbb{R}$, $X_{i}^{\alpha}\in \mathbb{S}(n)$ such that $b_1^{\alpha}+b^{\alpha}_2=0$,
\begin{equation*}
(b_{1}^{\alpha},\alpha(\tilde{x}_{1}^{\alpha}-\tilde{x}_{2}^{\alpha}),X_{1}^{\alpha}%
)\in \mathcal{\bar{P}}^{2,+}\widetilde{v}_{1}(t^{\alpha},\tilde{x}_{1}^{\alpha}), \ \ (b_{2}^{\alpha},\alpha(\tilde{x}_{2}^{\alpha}-\tilde{x}_{1}^{\alpha}),X_{2}^{\alpha}%
)\in \mathcal{\bar{P}}^{2,+}\widetilde{v}_{2}(t^{\alpha},\tilde{x}_{2}^{\alpha}),
\end{equation*}
and
\begin{equation*}
 \left(
\begin{array}
[c]{cccc}%
X_1^{\alpha}  & 0\\
0  & X_2^{\alpha}
\end{array}
\right)  \leq 3\alpha\left(
\begin{array}
[c]{cccc}%
I_n  & -I_n\\
-I_n  & I_n
\end{array}\right).
\end{equation*}
Moreover, it follows from equation \eqref{ineq-c} that
\begin{align*}
&b_{1}^{\alpha}+\lambda\widetilde{v}_{1}(t^\alpha,\tilde{x}_1^{\alpha}) -\widetilde{G}^*_{1}(\tilde{x}_1^{\alpha},\widetilde{v}_{1}(t^\alpha,\tilde{x}_1^{\alpha}),\alpha(\tilde{x}_{1}^{\alpha}-\tilde{x}_{2}^{\alpha}),X_1^{\alpha})\leq-\bar
{\delta}/T^{2},\\
&b_{2}^{\alpha}+\lambda\widetilde{v}_{2}(t^\alpha,\tilde{x}_2^{\alpha}) -\widetilde{G}^*_{2}(\tilde{x}_2^{\alpha},\widetilde{v}_{2}(t^\alpha,\tilde{x}_2^{\alpha}),\alpha(\tilde{x}_{2}^{\alpha}-\tilde{x}_{1}^{\alpha}),X_2^{\alpha})\leq-\bar
{\delta}/T^{2}.
\end{align*}

According to the definition of $\widetilde{G}^*_i$ and with the help of conditions \eqref{myw504} and \eqref{myw505}, we derive that
\begin{align*}
-2\bar
{\delta}/T^{2}&\geq\lambda\widetilde{v}_{2}(t^{\alpha},\tilde{x}_{2}^{\alpha})+\widetilde{G}^*(\tilde{x}_{2}^{\alpha},-\widetilde{v}_{2}(t^{\alpha},\tilde{x}_{2}^{\alpha}),\alpha(\tilde{x}_{1}^{\alpha}-\tilde{x}_{2}^{\alpha}),-X_{2}^{\alpha})\\
&\ \ \ \ \ +\lambda\widetilde{v}_{1}(t^{\alpha},\tilde{x}_{1}^{\alpha})-\widetilde{G}^*(\tilde{x}_{1}^{\alpha},\widetilde{v}_{1}(t^{\alpha},\tilde{x}_{1}^{\alpha}),\alpha(\tilde{x}_{1}^{\alpha}-\tilde{x}_{2}^{\alpha}),X_{1}^{\alpha})\\
&  \geq
\widetilde{G}^*(\tilde{x}_{2}^{\alpha},\widetilde{v}_{1}(t^{\alpha},\tilde{x}_{1}^{\alpha}),\alpha(\tilde{x}_{1}^{\alpha}-\tilde{x}_{2}^{\alpha}),-X_{2}^{\alpha})
-\widetilde{G}^*(\tilde{x}_{1}^{\alpha},\widetilde{v}_{1}(t^{\alpha},\tilde{x}_{1}^{\alpha}),\alpha(\tilde{x}_{1}^{\alpha}-\tilde{x}_{2}^{\alpha}),X_{1}^{\alpha})\\
&
\geq-C(L_1,L_2,\eta)(1+|\tilde{x}_{1}^{\alpha}|^2+|\tilde{x}_{2}^{\alpha}|^2)(|\widetilde{v}_{1}(t^{\alpha},\tilde{x}_{1}^{\alpha})||\tilde{x}_{1}^{\alpha}-\tilde{x}_{2}^{\alpha}|+\alpha |\tilde{x}_{1}^{\alpha}-\tilde{x}_{2}^{\alpha}|^2).
\end{align*}
The right-hand side tends to zero as $\alpha\rightarrow \infty$, which induces a contradiction. Consequently, we get that $\widetilde{v}^1\leq \widetilde{v}^2,$ the proof is complete.
\end{proof}

\end{document}